\documentclass{amsart}

\usepackage[numbers]{natbib}
\usepackage{graphicx}
\usepackage{amsmath}
\usepackage{amssymb}
\usepackage{amsthm}
\usepackage{libertine}
\usepackage{tikz-cd}
\usepackage{stmaryrd}
\usepackage{rsfso}
\usepackage{appendix}
\usepackage[colorlinks,linkcolor=red,citecolor=violet]{hyperref} 

\newtheorem{theorem}{Theorem}[section]
\newtheorem{fact}[theorem]{Fact}
\newtheorem{lemma}[theorem]{Lemma}

\newtheorem{corollary}[theorem]{Corollary}

\newtheorem{proposition}[theorem]{Proposition}

\theoremstyle{definition}
\newtheorem*{notation}{Notation}
\newtheorem{examplex}[theorem]{Example}
\newtheorem{definition}[theorem]{Definition}
\newtheorem{remarkx}[theorem]{Remark}
\newenvironment{example}
  {\pushQED{\qed}\examplex}{\popQED\endexamplex}
\newenvironment{remark}
  {\pushQED{\qed}\remarkx}{\popQED\endremarkx}

\newcommand{\cons}{\operatorname{\vdash}}
\newcommand{\conss}{\operatorname{\Vdash}}
\newcommand{\Eq}{\mb{Eq}}
\newcommand{\Seq}{\mb{Seq}}

\newcommand{\eff}{\Leftrightarrow}
\newcommand{\ef}{\text{ iff }}

\newcommand{\mc}[1]{\mathcal{#1}}

\newcommand{\mb}[1]{\mathbf{#1}}
\newcommand{\mbb}[1]{\mathbb{#1}}
\newcommand{\ms}[1]{\mathsf{#1}}

\newcommand{\mr}[1]{\mathrm{#1}}

\newcommand{\Set}{\mathbf{Set}}
\newcommand{\Poly}{\mb{Poly}}
\newcommand{\Mon}{\mb{Mnd}}
\newcommand{\SupL}{\mb{SupL}}
\newcommand{\CompL}{\mb{CompL}}
\newcommand{\catC}{\mc C}
\newcommand{\catD}{\mc D}
\newcommand{\SAlg}{\Sigma\text{-}\mr{Alg}}
\newcommand{\Alg}[1]{#1\text{-}\mr{Alg}}
\newcommand{\fml}{\mr{Fml}}

\newcommand{\Quo}{\mb{Quo}}
\newcommand{\MnF}{\mb{MnF}}

\newcommand{\set}[1]{\{#1\}}

\newcommand{\scomp}[2]{\operatorname{\{}#1 \mid #2\operatorname{\}}}

\newcommand{\op}{^{\operatorname{op}}}

\newcommand{\surj}{\twoheadrightarrow}
\newcommand{\id}{\operatorname{id}}
\newcommand{\pair}[1]{\langle #1 \rangle}
\newcommand{\nt}{\Rightarrow}

\newcommand{\gb}{\rightleftarrows}

\newcommand{\ass}[1]{\llbracket#1\rrbracket} 
\newcommand{\ov}[1]{\overline{#1}}
\newcommand{\qsi}[1]{\widetilde{#1}}

\newcommand{\inv}{^{-1}}

\newcommand{\power}[1]{\mb P(#1)}

\newcommand{\cps}{\fatsemi} 
\newcommand{\hook}{\hookrightarrow}

\newcommand{\ar}{\ms{ar}}
\newcommand{\yon}{\ms y}
\DeclareRobustCommand{\dashV}{\text{\reflectbox{$\Vdash$}}}
\DeclareRobustCommand{\dashv}{\text{\reflectbox{$\vdash$}}}
\newcommand{\cceff}{\operatorname{\dashV\!\Vdash}}
\newcommand{\ceff}{\operatorname{\dashv\!\vdash}}
\newcommand{\sqt}{\Mapsto}

\newcommand{\modall}{\operatorname{\triangledown}}

 \newsavebox{\possibility}
 \sbox\possibility{%
 \begin{tikzpicture}
 \draw (0,0ex) -- (-0.6ex,0.6ex);%
 \draw (0,0ex) -- (0.6ex,0.6ex);%
 \draw (-0.6ex,0.6ex) -- (0ex,1.2ex);%
 \draw (0.6ex,0.6ex) -- (0ex,1.2ex);%
 \end{tikzpicture}}

\begin{document}

\title{Syntax and Consequence Relation --- A Categorical Perspective}
\author{Lingyuan Ye}

\date{\today}

\address{Tsinghua University}
\email{\textrm{ye.lingyuan.ac@gmail.com}}

\maketitle


\section{Introduction}
\label{sec:intro}
On the proof-theoretic side, logic, roughly speaking, is mainly about the grammar of the language (syntax), and reasoning on this language (consequence relations). On the model-theoretic side, we further provides mathematical structures that evaluates the language (semantic models).

Among these, syntax is perhaps the easiest part. What one usually does to specify the syntax is to first fix a set of variables $X$, which almost always is chosen to be a countably infinite set, and then define the set of well-formed formulas $\fml$ with variables being in $X$. Here in this paper we will confine ourselves to only consider language of algebraic nature. This means that our signature for the language would be algebraic, and the only formula-forming rules would be application of function symbols. Syntax in richer context with variable bindings could become much less trivial.\footnote{See \cite{halbachleigh2021} for example for an extensive study of a very rich syntax theory in a first-order setting.}

The more interesting part of logic in our setting is how to reason with the given language, and how we provide the semantics. For the proof-theoretic side, the inferential structure on the specified language is formally described by \emph{consequence relation}. There are different forms of consequence relations appearing in the literature, depending on your specific proof systems, but almost all the examples can be generally viewed as some \emph{binary relations} between the sets $\fml$, $\power{\fml}$, $\power{\fml \times \fml}$, or possibly some other sets constructed from the syntactical data. Here $\power Y$ denotes the power set of $Y$. In every concrete scenario, a consequence relation is always recursively generated from a set of inference rules, which means that they are consequence relations that possess some \emph{finitary} nature.

Tarski is arguably one of the very first logicians to initiate the abstract study of the structure of general consequence relations. In his paper \cite{tarski1928remarques} he describes consequence relations in terms of \emph{closure operators} on a power set of some set. In further developments of abstract consequence relations, and in the majority of concrete logical systems, the consequence relations are required to be \emph{structural}, i.e. they should be invariant under substitution of variables. Structural consequence relations have been used by Blok and Pigozzi in \cite{blok1989algebraizable} to study the \emph{algebraisation} of logics; various refined treatment and generalisations of this work has been provided by Block and J\'onsson \cite{blok2006equivalence}, and by Galatos and Tsinakis \cite{galatos&tsinakis2009consequence}. They provides natural links from the syntactic side to the semantic side. 

For usual logics of algebraic signature, including propositional logics and all kinds of modal logics, a semantic model can be viewed abstractly as a structure that provides \emph{valuation} of syntactic data. A truth table or a modal evaluation frame are structures that determine whether a propositional or modal formula is \emph{true} or \emph{false}, or has value 1 or 0, or even not necessarily 2-valued. The recursive nature of the the usual truth clauses when defining such a valuation in particular shows furthermore that the valuation function should \emph{preserve tha algebraic structure} of the syntax, i.e. they should be an algebraic homomorphism with respect to the signature of the syntax.

All of these aspects will be studied in this paper, but with a complete reformulation and conceptual generalisation using the language of category theory. The reasons we feel such an approach is needed are manyfold. Though the study of modern logic has proven to be quite successful, the current framework described above is somewhat frustrating in several aspects. This starts from the very beginning of the way of describing the syntax. There are no good criteria \emph{a priori} that determines the choice of the fixed set of variables.\footnote{I should comment here that the truth of this sentence should be conditional, depending on one's perspective of what logic is. In more philosophical or metamathematical contexts, especially when certain logic systems are used to provide the foundation of mathematics, we cannot build the logic system with very strong mathematical axioms. There the choice of the set of variables, though not completely determined, are at least confined by our epistemic and computational ability. Here is this paper we will never consider such a foundational view point towards logic. What we will give is a \emph{mathematical study} of logic, where the mathematics we use are possible grounded by some other (meta)logical systems serving as the foundation.} There is indeed, I believe, a common feeling among experts that such a choice is not particularly important, but there is no explicit explanation or systematic formulation of why the obtained results are independent from such a choice. Fixing a set of variables also requires extra work when we further develop the formulation of consequence relations. The description of structurality, or invariance under substitution of variables, is yet another level of complexity we have to add into our framework. In \cite{galatos&tsinakis2009consequence} for example, where they provide a more general framework of consequence relations, structurality is modelled using the language of modules over a residuated lattices and invariance under its actions. Furthermore, it seems to the author that in the literature there lacks sufficient formulation of general principles in defining the semantics.

As indicated in the title, our approach towards logic, even in the basic level of describing the syntax, is \emph{categorical}, or \emph{functorial}, which very nicely fixes the fore mentioned drawbacks. We have two main conceptual advantage in taking up this approach. On the pure syntactic level, a functorial formuation fixes the fore mentioned drawbacks of a unnatural choice of a fixed set of variables. Functoriality means that we are going to describe the construction of syntax as a \emph{functor} $\mb F$ that outputs the set of well-formed formulas $\mb FX$ as the underlying set of variables $X$ varies. This is a very natural move from the author's perspective, since the usual way of specifying the syntax of a logic very naturally gives us a functor. This functorial approach towards the syntax is studied in Section~\ref{sec:synmonad}.

The more important reason that we provide a functorial formulation of syntax using $\mb F$ is that substitution of variables are automatically described by the action of $\mb F$ on maps. This means that the functor $\mb F$ describes the construction of syntactic data, as well as the substitutional structure, within a single integrated package. We will then apply it to further study abstract consequence relations, where structurality is now built in, and general semantics. There we will state and show our other general perspective towards logic: consequence relations and semantics are general ways to construct \emph{quotients} of structures. The details will be more specifically given in later section. As we will see, our approach provides a very clean and conceptual understanding of the nature of logic (viewed as a branch of mathematics), and extends and generalises various existing results, concerning consequence relations and algebraisation, algebraic semantics, etc., that already established in the literatures, as well as derives several new ones.


This is the second draft of a longer text, which we will progressively release new versions, including additions, revisions and updates. We assume the readers to be familiar with basic categorical notions, including categories, functors, natural transformations, adjunctions, limits and colimits, etc..

\begin{notation}
  To avoid notation clashes, in this paper we will use $\cps$ to denote function (morphisms) composition, and we will reserve the symbol $\circ$ for the monoidal product on the category of endo-functors given by functor composition. We will also often omit it if there's no confusion. And unlike the usual order for function composition, $f \cps g$ will denote post-compose $g$ with $f$, i.e. $f \cps g = g \circ f$ in the usual notation for function composition. For representable functors, we will use $\yon^X$ to denote the Hom-functor $\Set(X,-)$ on $\Set$ (or any other arbitrary category $\catC$). In particular, we will use $\yon \cong \yon^1$ to denote the identify endo-functor, with 1 denotes the singleton set. The functor category from $\catC$ to $\catD$ will be denoted as $[\catC,\catD]$. Given any binary relation $R$ on a set $X$, we use $x R y$ to denote $(x,y) \in R$.
\end{notation}

\section{Syntactic Monad}
\label{sec:synmonad}
In this paper, we will mainly focus on logic over algebraic (propositional) signatures. The syntax over such a signature is quite simple, since it does not involve variable binding. The syntax of such logics can then be studied from a universal algebra point of view, using the language of monads and algebras of monads.

Let a signature $\Sigma$ be a set of connectives, together with a map $\ar : \Sigma \to \omega$ assigning to each connective its arity. We identify a connective $\star$ that has arity zero as a (propositional) constant. For example, in the usual syntax of propositional logic both $\top$ and $\bot$ are propositional constants. The signature here may contain the usual propositional connectives, including $\top,\bot,\wedge,\vee,\neg,\to$ and others, or modal operators $\modall$ of arity larger than zero as well. Hence, such a setting at least encompass all standard (multi-)modal logics. For any set of variables $X$, the set of well-defined formulas ($\Sigma$-terms) over $X$ are usually defined recursively as follows,
\[ \varphi :\equiv x \in X \mid \star(\varphi_1,\cdots,\varphi_{n}) \]
where $\star$ ranges over connectives in $\Sigma$, and $n$ is the arity of $\star$. We denote this set of well-formed formulas as $\mb FX$. If a formula $\varphi$ can be constructed in the above sense, viz. $\varphi\in\mb FX$, then it only contains variables in the set $X$. 

The above description of syntax actually extends to an endo-functor on $\Set$
\[ \mb F : \Set \to \Set. \]
For any function $f : X \to Y$, we think of it as specifying variable substitution, taking $x\in X$ into $f(x)\in Y$. The function $\mb Ff$ then sends each formula $\varphi$ in $\mb FX$ into a formula in $\mb FY$, by \emph{simultaneously substituting} every $x\in X$ in $\varphi$ into $f(x)$. We will also write $\varphi^f$, the substitution of $\varphi$ along $f$, to denote $\mb Ff(\varphi)$. For a subset $\Phi$ of $\mb FX$, we write $\Phi^f$ to denote the set $\scomp{\varphi^f}{\varphi\in\Phi}$ as well. The fact that the functor $\mb F$ acts on morphisms models substitution is very important to our categorical approach of logic.

\begin{remark}
\label{rem:synmonsubst}
  Our functorial approach here subsumes the usual formulation of substitution in the literature. Usually in other literatures, where we only describe the syntax as a set $\fml$ based on a fixed set of variables $X$, substitution is modelled by action of a monoid $M$ on $\fml$. From our functorial perspective, the set of formulas $\fml$ is simply the set $\mb FX$ for the chosen variable set $X$. The monoid that describes substitution is actually isomorphic to the monoid of endo-functions on $X$
  \[ M \cong \Set(X,X). \]
  The monoid structure on $\Set(X,X)$ is given by composition of functions. The functor $\mb F$ then describes the action of this monoid on our set of formulas simply as follows: For any given $\sigma\in\Set(X,X)$ and any $\varphi\in\mb FX$,
  \[ \sigma\cdot\varphi = \mb F\sigma(\varphi) = \varphi^\sigma. \]
  The fact that $\mb F$ is a functor ensures this is a well-defined action of a monoid
  \[ \varphi^{\id_X} = \mb F\id_X(\varphi) = \id_{\mb FX}(\varphi) = \varphi, \]
  \[ \varphi^{\sigma\cps\tau} = (\mb F\sigma \cps \mb F\tau)(\varphi) = (\varphi^\sigma)^\tau. \]
  Hence, the usual formulation of substitution on the syntactical level is subsumed in our functorial description.
\end{remark}

The crucial observation is that $\mb F$ further has a monad structure that, as we will show in Lemma~\ref{lem:fsigalg} below, describes the $\Sigma$-term algebras associated to the syntax. Such a result belongs to a much more general connection between algebras and monads, a description of which can be found in \cite{hyland&poewr2007algandmonad}.\footnote{Be careful that in \cite{hyland&poewr2007algandmonad}, the notion of algebraic theory is slightly different than usual; it is formulated in the representation-invariant categorical notion of theories, viz. Lawvere theories. See the reference there for more information.}

A monad on $\Set$ is a triple $(\mb T,\eta,\mu)$, where $\mb T$ is an endo-functor on $\Set$, and $\eta,\mu$ are natural transformations of the following type,
\[ \eta : \yon \to \mb T, \quad \mu : \mb T\mb T \to \mb T, \]
Recall that we use $\yon$ to denote the identity functor. They are required to make the following diagrammes commute,
\[
\begin{tikzcd}
  \mb T \ar[dr, equal] \ar[r, "\eta \circ \id_\mb T"] & \mb T \mb T \ar[d, "\mu"] & \mb T \ar[l, "\id_\mb T \circ \eta"'] \ar[dl, equal] \\
  & \mb T
\end{tikzcd} \quad
\begin{tikzcd}
  \mb T\mb T\mb T \ar[r, "\mu \circ \id_\mb T"] \ar[d, "\id_\mb T \circ \mu"'] & \mb T\mb T \ar[d, "\mu"] \\
  \mb T\mb T \ar[r, "\mu"'] & \mb T
\end{tikzcd}
\]
The left, right triangle and the square are called the left, right unit law and associativity law, respectively. Equivalently, let $([\Set,\Set],\circ,\yon)$ be the monoidal category of endo-functors on $\Set$ with the monoidal product given by functor composition; then a monad on $\Set$ is precisely an \emph{internal monoid} in this monoidal category.\footnote{Given a monoidal category $(\catC,\otimes,I)$, an internal monoid is an object $x$ equipped with two maps $u : I \to x$, $m : x \otimes x \to x$, which are further required to satisfy the usual unit and associativity laws expressed in equations between composites of functions. In particular, an internal monoid in $(\Set,\times,1)$ is simply a monoid in the usual sense. See \cite[Section III.6]{maclane2013categories} for a very brief account of internal groups, which are internal monoid equipped with an additional inverse map; see \cite{borceux2005internalaction} for a much more detailed discussion of internal algebras in a category.} The above diagrammes precisely express that $\eta$ is a two-sided unit of $\mu$ and $\mu$ is associative. We refer the readers to \cite[Chapter~VI]{maclane2013categories} for a more detailed description.

Explicitly, we have the following morphisms in $[\Set,\Set]$
\[ \eta_\mb F : \yon \to \mb F,\ \mu_\mb F : \mb F \mb F \to \mb F. \]
For any set $X$, we have the component of the unit
\[ \eta_{\mb F,X} : X \to \mb FX. \]
It is given by the inclusion $X \hook \mb F(X)$, since in our description of syntax we have defined that every variable (element) in $X$ is a well-formed formula. To avoid confusion, when we want to consider $x\in X$ as a formula in $\mb FX$ we will always write $\eta_{\mb F,X}(x)$ explicitly. We also have the component of multiplication
\[ \mu_{\mb F,X} : \mb F\mb FX \to \mb FX. \]
The elements in $\mb F\mb FX$ are formulas constructed from the set of variables being the set of formulas $\mb FX$ over $X$, which can be naturally viewed as formulas over $X$ itself. Formally, we can inductively define $\mu_{\mb F,X}$ as follows: Let $\varphi$ ranges over $\mb FX$ and $\psi_j$ ranges over $\mb F\mb FX$, then we have
\begin{align*}
  \mu_{\mb F,X}(\eta_{\mb F,\mb FX}(\varphi)) &:= \varphi, \\
  \mu_{\mb F,X}(\star_i(\psi_1,\cdots,\psi_{n_i})) &:= \star(\mu_{\mb F,X}(\psi_1),\cdots,\mu_{\mb F,X}(\psi_{n})),
\end{align*}
for any $\star$ in $\Sigma$ of arity $n$. It is easy to check the naturality for both $\eta_{\mb F}$ and $\mu_{\mb F}$.

\begin{lemma}
  $(\mb F,\eta_\mb F,\mu_\mb F)$ is a monad.
\end{lemma}
\noindent
When no confusion would arise, we usually use the functor part to denote the whole monad structure. We call $\mb F$ the \emph{syntactic monad}.

Monads have algebras. Given the monad $\mb F$, there is an induced category of algebras $\Set^\mb F$ of $\mb F$, whose objects are maps of the form $\alpha : \mb F X \to X$, where $X$ is called the \emph{carrier} of $\alpha$, that further makes the following diagrammes commute,
\[
\begin{tikzcd}
  X \ar[dr, equal] \ar[r, "\eta_{\mb F,X}"] & \mb FX \ar[d, "\alpha"] \\
  & X
\end{tikzcd} \quad
\begin{tikzcd}
  \mb F\mb FX \ar[d, "\mb F\alpha"'] \ar[r, "\mu_{\mb F,X}"] & \mb FX \ar[d, "\alpha"] \\
  \mb FX \ar[r, "\alpha"'] & X
\end{tikzcd}
\]
We will use $\alpha$, as well as its carrier $X$ when there no confusion would arise, to denote an $\mb F$-algebra. A morphism between two algebras $\alpha : \mb FX \to X$ and $\beta : \mb FY \to Y$ is a function $f : X \to Y$ that makes the following diagramme commute,
\[
\begin{tikzcd}
  \mb FX \ar[r, "\mb Ff"] \ar[d, "\alpha"'] & \mb FY \ar[d, "\beta"] \\
  X \ar[r, "f"'] & Y
\end{tikzcd}
\]
For example, for any set $X$ the multiplication $\mu_{\mb F,X} : \mb F\mb FX \to \mb FX$ is actually an $\mb F$-algebra. It makes the corresponding diagrammes commute due to unit and associative laws of monad. This extends to an adjunction
\[ \ms F : \Set \gb \Set^\mb F : U_\mb F, \]
where $\mb F$ is the induced monad of this adjunction. $U_\mb F$ is the forgetful functor sending each $\mb F$-algebra to its carrier. Such an adjunction in particular means that the $\mb F$-algebra structure $\mu_{\mb F,X}$ on $\mb FX$ is the \emph{free $\mb F$-algebra} on $X$. Again, see \cite[Chapter.~VI]{maclane2013categories} for a more detailed description of algebras of a monad.

In the usual terminology of logic literatures, $\mb F$-algebras are precisely the term algebras for the given signature $\Sigma$, as we will show below. A $\Sigma$-algebra is a set $X$ equipped with operations $\star^X : X^n \to X$, for any $\star\in\Sigma$ that $\ar(\star)=n$. When $n = 0$ we identify $X^0$ as a singleton set $1$. A $\Sigma$-algebra morphism between two $\Sigma$-algebras is a map $f : X \to Y$ that preserves these operations, i.e. for any $\star\in\Sigma$ we have
\[ f(\star^X(x_1,\cdots,x_n)) = \star^Y(f(x_1),\cdots,f(x_n)). \]
There is then an evident category of $\Sigma$-algebras, which we denote as $\SAlg$. The following result is a consequence of a much more general correspondence between monads and algebraic theories.
\begin{lemma}
\label{lem:fsigalg}
  The category $\Set^\mb F$ and $\SAlg$ are equivalent (isomorphic).
\end{lemma}
\begin{proof}[Proof Sketch]
  We only describes how $\mb F$-algebras and $\Sigma$-algebras corresponds bijectively to each other. Given a $\Sigma$-algebra $X$, we can easily construct an $\mb F$-algebra
  \[ \gamma : \mb FX \to X, \]
  by defining inductively as follows
  \begin{align*}
    \gamma(\eta_{\mb F,X}(x)) &:= x, \\
    \gamma(\star(\varphi_1,\cdots,\varphi_n)) &:= \star^X(\gamma(\varphi_1),\cdots,\gamma(\varphi_n)),
  \end{align*}
  for any $\star\in\Sigma$. One can then prove inductively that
  \[ \eta_{\mb F,X} \cps \gamma = \id_X, \quad \mu_{\mb F,X} \cps \gamma = \mb F\gamma \cps \gamma, \]
  which means $\gamma$ is in fact an $\mb F$-algebra.

  On the other hand, given an $\mb F$-algebra $\alpha : \mb FX \to X$ we associate it with a $\Sigma$-algebra structure. For any $\star\in\Sigma$ with $\ar(\star)=n$, we can simply define
  \[ \star^X_\alpha(x_1,\cdots,x_n) := \alpha(\star(\eta_{\mb F,X}(x_1),\cdots,\eta_{\mb F,X}(x_n))). \]
  A further inductive proof shows that the above two processes are mutually inverse of each other, and under this bijective correspondence the notions of $\mb F$-algebra and $\Sigma$-algebra homomorphisms coincide.
\end{proof}

\begin{remark}
\label{rem:synpolymonad}
  Lemma~\ref{lem:fsigalg} actually implies that $\mb F$ is a \emph{free monad}, by the fact that there is another description of $\Sigma$-algebras using a polynomial functor. A polynomial functor is defined to be a functor in $[\Set,\Set]$ which is naturally isomorphic to a coproduct of representable functors. See \cite{gambino2013polynomial} for an introduction. Let $H : \Set \to \Set$ be the following functor
  \[ H = \sum_{\star\in\Sigma}\yon^{\ar(\star)}. \]
  where $\sum$ denotes the coproducts (disjoint unions). By definition $H$ is a polynomial functor, and it sends every set $X$ to the set $\sum_{\star\in\Sigma}X^{\ar(\star)}$. There is a similar notion of algebras for an arbitrary endo-functor on $\Set$, not only just for monads: An $H$-algebra is simply a set map $\gamma : HX \to X$ satisfying no further conditions, and there is an evident notion of morphisms between $H$-algebras. Hence, we have a category of $H$-algebras, which we denote as $\Alg H$. In this case, an $H$-algebra on $X$ is a function
  \[ \sum_{\star\in\Sigma} X^{\ar(\star)} \to X, \]
  which is equivalent to give for every $\star\in\Sigma$ an operation $X^n \to X$ where $n=\ar(\star)$. This means the notion of $\Sigma$-algebras is the same as $H$-algebras, and we then have yet a further equivalence (isomorphism)
  \[ \Set^\mb F \cong \Alg H. \]
  Such an equivalence in particular shows that $\mb F$ is the \emph{free monad} on the polynomial functor $H$, which implies that $\mb F$ itself is a polynomial monad, i.e. the functor $\mb F$ is also a polynomial functor such that the unit and multiplication of its monad structure are all cartesian. See \cite{kelly1980unifiedfree} for a detailed technical background on free monads and related conceptions. We will come back to this point in later sections. Such an observation makes it possible to further pursue a purely abstract description of syntactic monads based on free monads over a polynomial functor. We leave this for future works.
\end{remark}

We end this section by discussing finite products of $\mb F$. In later sections, besides mere formulas, we will also be interested in pairs of formulas, or more generally a sequent of some type. We can model these in our categorical framework by taking products of $\mb F$ in $[\Set,\Set]$. Now limits and colimits in a functor category like $[\Set,\Set]$ are taken component-wise; see \cite[Chapter.~2.15]{borceux1994handbook1} on limits and colimits in a functor category. This in particular means that for the product $\mb F^n$ in $[\Set,\Set]$ and for any set $X$, we have
\[ \mb F^n(X) = (\mb FX)^n. \]
Now the action of $\mb F$ on functions is also component-wise. For a tuple $(\varphi_1,\cdots,\varphi_n) \in \mb F^nX$, we have
\[ (\varphi_1,\cdots,\varphi_n)^f = \mb F^nf(\varphi_1,\cdots,\varphi_n) = (\mb Ff\varphi_1,\cdots,\mb Ff\varphi_n) = (\varphi_1^f,\cdots,\varphi_n^f). \]
This shows that $\mb F^n$ models point-wise substitution for tuples of formulas, which is exactly what we want.

All the above results have shown us that the syntactic monad $\mb F$ indeed contains all the relevant information about the syntax: As far as the single functor $\mb F$ goes, the set $\mb FX$ gives out the set of well-formed formulas constructed from the variable set $X$ for the given signature $\Sigma$, and the functor $\mb F$ acting on set maps $f : X \to Y$ models uniform substitution. This also works nicely if we consider products of $\mb F$. The monad structure on $\mb F$ further contains all the information about $\Sigma$-algebras. Lemma~\ref{lem:fsigalg}, together with the free $\mb F$-algebra functor $\ms F$, establishes the familiar fact that $\mb FX$ is the free $\Sigma$-algebra on the generating set $X$.

\section{Abstract Consequence Relations}
\label{sec:acrsup}
The syntactic monad $\mb F$ describes, by its name, the syntax of our logic, which is only establishing the ground work for the more important part of logic, viz. inferential and reasoning structures and semantic relations. In this section, we will focus on the proof-theoretic side of the story.

Reasoning in logic is usually represented abstractly in the form of consequence relations. From a non-categorical setting, with a fixed set of variables $X$ and hence a fixed set of formulas $\fml$, a consequence relation can be defined on various sets constructed from $\fml$, depending on the style of the inference system. An \emph{asymmetric} consequence relation is usually considered as a binary relation $\cons \subseteq \mb P\fml \times \fml$. A \emph{symmetric} consequence relation is a binary relation $\cons \subseteq \mb P\fml \times \mb P\fml$. It is well known that asymmetric and symmetric consequence relations are equivalent, and hence we will only consider symmetric ones in the remainder. Consequence relations on other sets are also studied. To provide an algebraic perspective towards logics and study their algebraisations we also want to consider consequence relations defined on pairs of formulas $\mr{Eq} = \fml \times \fml$, viz. $\cons \subseteq \mb P\mr{Eq} \times \mb P\mr{Eq}$. More generally, we may consider consequence relations defined on sequents. A sequent of type $(n,m)$ with $n,m\in\omega$ is usually denoted as
\[ \varphi_1,\cdots,\varphi_n \sqt \psi_1,\cdots,\psi_m. \]
In particular, a formula can be identified with a $(0,1)$-sequent, and a pair of formulas can be identified as a $(1,1)$-sequent. The set of all $(m,n)$-sequents $\mr{Seq}_{n,m}$ is simply given by $\fml^n \times \fml^m \cong \fml^{n+m}$. Hence basically, we want to consider consequence relations defined on finite products of $\fml$ in general.

In \cite{galatos&tsinakis2009consequence}, Galatos and Tsinakis have generalised consequence relations to arbitrary complete lattices, which leads to a very nice mathematical theory of consequence relations. In this section we will follow such a spirit to give a mathematically rigorous theory of consequence relations, but with several changes of perspective.

Firstly, a categorical mind set has taught us that in most of the mathematical context we should put objects, as well as corresponding \emph{maps between objects}, into consideration. From our point of view, the theory of abstract consequence relations is much more naturally related to \emph{suplattices}, rather than complete lattices. A suplattice by definition is a complete join-semi lattice, viz. a poset that has arbitrary joins. It is well-known that every suplattice is automatically complete, hence object-wise complete lattices and suplattices are the same thing.\footnote{A proof can be found in \cite[p.~27]{johnstone1982stone}.} However, a morphism between suplattices is only required to preserves arbitrary joins, rather than all meets and all joins. This makes $\SupL$, the category of suplattices, have very different categorical property than $\CompL$, the category of complete lattices.\footnote{One way to see their difference is that suplattices are \emph{algebraic} over $\Set$. In later sections we will describe that $\SupL$ is equivalent to the category of $\mb P$-algebras, where $\mb P$ is the covariant power set monad. In particular, this implies that all free-suplattice-constructions exists in $\Set$, due to the free suplattice functor $\ms P$ induced by $\mb P$. However, complete lattices are \emph{not} algebraic over $\Set$. The free complete lattices on $n > 2$ generators does not exists (or must form a proper class); see \cite[p.~31]{johnstone1982stone} for a proof.}

Using the language of suplattices, we can see much more clearly that consequence relations in logic are natural ways to construct \emph{quotients}, as we've mentioned briefly in Section~\ref{sec:intro}. Furthermore, the category $\SupL$ is very closely connected to the covariant power functor $\mb P$. As we will show later, there is a monad structure on $\mb P$, and the category $\Set^\mb P$ of $\mb P$-algebras is the same as suplattices. This naturally establishes sets including $\mb P\fml,\mb P\mr{Eq}$, or more generally $\mb P\mr{Seq}_{n,m}$, where consequence relations are usually built on, as typical examples of free suplattices. We will see more of such connections in the future part of this section.

Secondly, continuing to follow the philosophy of the previous section, we want the description of substitution and structurality of consequence relations to be directly built within our framework. In \cite{galatos&tsinakis2009consequence}, besides using complete lattices to define abstract consequence relations, it further describes substitution as some module structures --- complete lattices with an action of a residuated lattice modelling substitution --- which is another level of complexity we prefer not to have. In the following texts, we will again directly adopt a functorial approach to express structurality. These considerations have led to the following formulation.

We first give a brief introduction to suplattices, in particular morphisms between suplattices. By the adjoint functor theorem, a monotone map $f^* : P \to Q$ between two suplattice is a suplattice morphism, i.e. it preserves arbitrary joins, if and only if there is another monotone map $f_* : Q \to P$ such that $f^*$ is left adjoint to $f_*$
\[ f^* : P \gb Q : f_*; \]
it means for any $x\in P,y\in Q$ we have
\[ f^*(x) \le y \eff x \le f_*(y). \]
When $f^*$ is a suplattice morphism, $f_*$ can be explicitly described by the following formula
\[ f_*(y) = \bigvee\scomp{x\in P}{f^*(x) \le y}. \]
Similarly, we can express the left adjoint using the right adjoint
\[ f^*(x) = \bigwedge\scomp{y \in Q}{f_*(y) \ge x}. \]
We leave for the readers to check that $f^*,f_*$ indeed form an adjunction in both cases. This means that to show a monotone function is a suplattice morphism, it suffices to find a right adjoint. The right adjoint $f_*$ in general will \emph{not} be a suplattice morphism; it preserves arbitrary meets but not joins by general category theory. However, when $f^*$ is an isomorphism, $f_*$ will simply be the inverse of $f^*$ and both $f^*$ and $f_*$ will preserve arbitrary joins and meets. 

Notice that for a left adjoint $f^*$, it is injective if and only if it is an order embedding, viz. $x \le y \eff f^*(x) \le f^*(y)$. The same holds for a right adjoint. The following lemma shows that the right and left adjoint are closely connected.
\begin{lemma}
\label{lem:surinclus}
  Given a morphism between suplattices $f^* : P \to Q$, it is surjective (resp. injective) if and only its right adjoint $f_*$ is injective (resp. surjective).
\end{lemma}
\begin{proof}
  We only show $f^*$ being surjective (resp. injective) implies $f_*$ being injective (resp. surjective). The other way around is similar and we leave for the readers to check.

  Suppose $f^*$ is surjective, since for any $y\in Q$ there exists some $x$ that $f^*(x) = y$, it follows that
  \[ f_*(y) = \bigvee\scomp{x\in P}{f^*(x) \le y} = \bigvee_{f^*(x) = y} x. \]
  We then further have
  \[ f^*f_*(y) = \bigvee_{f^*(x) = y} f^*(x) = y. \]
  This means $f_*$ must be injective. On the other hand, suppose $f^*$ is injective. For any $x,y\in P$, $x \le y \eff f^*(x) \le f^*(y)$. By the adjunction this further implies that
  \[ y \le f_*f^*(x) \eff f^*(y) \le f^*(x) \eff y \le x. \]
  Hence, $f_*f^*(x) = x$, which shows that $f_*$ must be surjective.
\end{proof}
\noindent
We refer the readers to \cite[Chapter.~I]{joyal1984extension} for a more complete formulation of the categorical properties of the category $\SupL$.

To work our step up to a fully functorial approach of abstract consequence relations, we first reformulate some of the results presented in \cite{galatos&tsinakis2009consequence} using suplattices. As we will show below, Lemma~\ref{lem:bijecpoint} is a strong evidence that considering suplattices is far more natural than considering complete lattices, and it gives us a first instance of our general philosophy that consequence relations are general ways to construct \emph{quotients} of structures.

The following definition of consequence relation on a single suplattice is adapted from \cite{galatos&tsinakis2009consequence}. We define an abstract consequence relation on a suplattice $P$ to be a \emph{preorder} $\cons$ on $P$, such that $\cons \supseteq \operatorname\ge$, and for any $x\in P$
\[ x \cons \bigvee\scomp{y\in P}{x \vdash y}. \]
We say \emph{$x$ implies $y$} if $x \cons y$ in $P$.

Another useful notion is that of a \emph{closure operator}, which corresponds to Tarski's original approach to abstract consequence relations in \cite{tarski1928remarques}. A closure operator $j$ on a suplattice $P$ is a \emph{monotone} function $j : P \to P$ which is \emph{extensive}, viz. $j(x) \ge x$ for any $x$, and \emph{idempotent}, viz. $jj(x)=j(x)$ for any $x$.

Finally, we present what we mean by a \emph{quotient} of a suplattice. A quotient of $P$ is a surjective suplattice homomorphism $e^* : P \surj Q$, considered \emph{up to isomorphism}. In other words, a quotient strictly speaking is an equivalence class of surjections out of $P$ in the category $\SupL$. Two surjections $e^*_1 : P \surj Q_1,e^*_2 : P \surj Q_2$ are considered equivalent if and only if there is an isomorphism $Q_1 \cong Q_2$ between suplattices making the following diagramme commute,
\[
\begin{tikzcd}
  & P \ar[dr, two heads, "e_2^*"] \ar[dl, two heads, "e_1^*"'] & \\
  Q_1 \ar[rr, "\cong"'] & & Q_2
\end{tikzcd}
\]
In the remaining texts we will loosely speak of a surjection as a quotient, but silently assuming any two isomorphic surjections in the above sense represent the same quotient.

From our knowledge the following lemma, which states that consequence relations and closure operators on a suplattice are the same as quotients, is at least folklore, if not well-known. The piece that consequence relations and closure operators on a suplattice corresponds bijectively is implicit in \cite{galatos&tsinakis2009consequence}. The fact that closure operators and quotients are the same is perhaps known to the experts for a much longer period --- a proof can be found in \cite[Section~I.4]{joyal1984extension}; a similar result for frames, left exact versions of suplattices, is also established in \cite[Section~II.2]{johnstone1982stone}. For the convenience of the readers we collect the pieces together and prove the following lemma.
\begin{lemma}
\label{lem:bijecpoint}
  Consequence relations, closure operators, and quotients of a suplattice bijectively correspond to each other.
\end{lemma}
\begin{proof}
  (1): We first show that consequence relations and closure operators on $P$ are bijectively correspondent. Given a closure operator $j$, we define $\cons_j$ to be
  \[ x \cons_j y \eff y \le j(x). \]
  It is a preorder:
  \[ x \le j(x) \nt x \cons_j x; \]
  \[ x \cons_j y, y \cons_j z \nt z \le j(y) \le jj(x) = j(x) \nt x \cons_j z. \]
  It contains $\ge$:
  \[ y \le x \nt y \le x \le j(x) \nt x \cons_j y. \]
  We also have
  \[ x \cons_j \bigvee_{x \cons_j y} y \eff \bigvee_{x \cons_j y}y \le j(x) \eff (x \cons_j y \nt y \le j(x)), \]
  Hence, $\cons_j$ is a well-defined consequence relation. On the other hand, given a consequence relation $\cons$, we define a closure operator $\gamma$ to be the following,
  \[ \gamma(x) = \bigvee_{x \cons y} y. \]
  It is obviously increasing because $\cons$ is reflexive. For monotonicity, suppose $x \le z$ and thus $z \cons x$. Given any $y$ that $x \cons y$, by transitivity we have $z \cons y$ as well, which implies 
  \[ \gamma(x) = \bigvee_{x \cons y} y \le \bigvee_{z \cons w} w = \gamma(z). \]
  Finally, to prove idempotence of $\gamma$ we show that
  \[ x \cons z \eff \gamma(x) \cons z. \]
  The left to right direction is trivial, since $\gamma(x) \ge x$ thus $\gamma(x) \cons x$. For the other direction we only need to observe that by definition $x \cons \gamma(x)$. Thus, we have
  \[ \gamma(\gamma(x)) = \bigvee_{\gamma(x) \cons y} y = \bigvee_{x\cons y} y = \gamma(x). \]
  Hence, $\gamma$ is a closure operator. Finally, it is easy to see that these operations are inverse to each other, since we have
  \[ x \cons_{\gamma} z \eff z \le \bigvee_{x \cons y} y \eff x \cons z, \]
  and also
  \[ \gamma_j(x) = \bigvee_{x \cons_j y} y = \bigvee_{y \le j(x)} y = j(x). \]

  (2): Next, we show that quotients of suplattices are in bijective correspondence to closure operators on $P$. Given a quotient
  \[ e^* : P \surj Q, \]
  we define $e : P \to P$ to be the following map
  \[ e(x) = e_*e^*(x). \]
  From Lemma~\ref{lem:surinclus} we know that for any $y \in Q$
  \[ e^*e_*(y) = y. \]
  This in particular implies that $e$ is a closure operator. On the other hand, given a closure operator, we defined a quotient
  \[ j^* : P \surj P_j, \]
  where $P_j$ is the set of fixed-points of $j$, or equivalently the image of $P$ under $j$, and
  \[ j^*(x) = j(x). \]
  By definition, $j^*$ is surjective. The right adjoint $j_*$ can simply taken to be the inclusion
  \[ j_* : P_j \hook P. \]
  They are indeed adjoint because for any $y\in P_j$, $j(x) \le y$ implies $x \le y$, and $x \le y$ implies $j(x) \le j(y) = y$, thus a closure operator indeed gives us a quotient. We show that these two operations are inverse to each other. It is obvious that the closure operator induced by the quotient $j^*$ is simply $j$. It remains to show that given a quotient $e^* : P \surj Q$, we must have $Q\cong P_e$. Observe the right adjoint $e_* : Q \to P$ actually restricts to a map
  \[ e_* : Q \to P_e. \]
  This is because for any $q\in Q$
  \[ ee_*(q) = e_*e^*e_*(q) = e_*(q). \]
  This restricted morphism is surjective, because for any $x \in P_e$ we have
  \[ x = e^*e_*(x). \]
  By Lemma~\ref{lem:surinclus} we also know $e_*$ is injective. $e_* : Q \to P$ is a right adjoint thus preserves meets; meets in $P_e$ are calculated the same as in $P$; hence the restricted map $e_* : Q \to P_e$ is indeed an isomorphism between posets, which in particular shows that $Q$ is isomorphic to $P_e$ as a suplattice.
\end{proof}

In the light of the above correspondence, given a quotient $j^* : P \surj P_j$ that corresponds to the closure operator $j$ and consequence relation $\cons$ on $P$, we call elements in $P_j$ as \emph{closed theories}, or simply say it is \emph{closed}, and also call the quotient suplattice $P_j$ \emph{the lattice of closed theories}. For any $x\in P$, its \emph{closure} is given by $j(x) = \bigvee_{x\cons y}y$, and every element implies its closure. This in particular shows that
\[ w \in P \text{ is closed } \ef \forall y \in P,\ w \cons y \nt y \le w. \]
In the below adjunction
\[ j^* : P \gb P_j : j_*, \]
$j^*$ is surjective and $j_*$ is an inclusion between posets. Since $j^*$ is a left adjoint thus preserves joins, it follows that joins in $P_j$ are computed as the closure of joins in $P$. Since $j_*$ is a poset-embedding and preserves arbitrary meets, $P_j$ must then be closed under arbitrary meets in $P$, and meets in $P_j$ are computed exactly the same as in $P$.

\begin{example}
\label{exp:imagefactor}
  Let's apply the above very useful lemma to a concrete example where we describe the image-factorisation of suplattices. As we will see later, the category $\SupL$ is \emph{algebraic} over $\Set$, i.e. there is a monad such that $\SupL$ is the category of algebras of this monad. In particular, this implies that we have image-factorisation in $\SupL$. We only need to observe that the construction of a closure operator induced by a surjection in the previous proof actually extends to arbitrary suplattice morphisms.

  Given any suplattice morphism
  \[ e^* : P \to Q, \]
  the map $x \mapsto e_*e^*(x)$ is indeed a closure operator. It is obviously monotone. $x \le e_*e^*(x)$ simply because $e_*$ is right adjoint to $e^*$ and $e^*(x)\le e^*(x)$. It is idempotent by noticing the following computation:
  \begin{align*}
    e^*e_*e^*(x)
    &= e^*\left(\bigvee\scomp{ y \in P }{e^*(y) \le e^*(x)}\right) \\
    &= \bigvee\scomp{e^*(y)}{e^*(y) \le e^*(x)} \\
    &= e^*(x)
  \end{align*}
  Let us denote this closure operator by $\epsilon$. We then have a factorisation of $e$ as follows,
  \[
  \begin{tikzcd}
    P \ar[rr, "e^*"] \ar[dr, two heads, "\epsilon^*"'] & & Q \\
    & P_\epsilon \ar[ur, tail, "e^*_\epsilon"']
  \end{tikzcd}
  \]
  The map $e^*_\epsilon$ sends any $x\in P_\epsilon$ to $e^*(x)$; in other words, it is the map $e^*$ restricted to $P_\epsilon$. The above diagramme indeed commutes
  \[ e_\epsilon^*\epsilon^*(x) = e^*e_*e^*(x) = e^*(x). \]
  The right adjoint of $e^*_\epsilon$ is given by $y \mapsto \epsilon^*e_*(y) = e_*e^*e_*(y)$, and we have
  \[ x \le e_*e^*e_*(y) \eff e^*(x) \le e^*e_*(y). \]
  By the adjunction we know that $e^*e_*(y)\le y$, hence $x \le e_*e^*e_*(y) \nt e_\epsilon^*(x) \le y$. On the other hand, suppose $e^*(x) \le y$ for $x\in P_\epsilon, y\in Q$. We further note
  \begin{align*}
    e^*(x) = e^*e_*e^*(x) \le y
    &\nt e_*e^*(x) \le e_*(y) \\
    &\nt e^*e_*e^*(x) = e^*(x) \le e^*e_*(y) \\
    &\nt x \le e_*e^*e_*(y).
  \end{align*}
  This proves that $y \mapsto e_*e^*e_*(y)$ is indeed the right adjoint of $e_\epsilon^*$. Hence, the above is indeed a commuting diagramme in $\SupL$. Finally, to show the above diagramme consists of the image-factorisation of $e^*$ we only need to observe that $e^*_\epsilon$ is injective. For any $x \in P_\epsilon$ we know that
  \[ e_*e^*(x) = x. \]
  This in particular shows that $e^*_\epsilon$, which is $e^*$ restricted to $P_\epsilon$, must be injective. Thus, $P_\epsilon$ is indeed the image of $e^*$, and the above diagramme depicts the image-factorisation.
\end{example}

The above is the whole story of how consequence relations relates to closure operators and quotients of suplattices. To further study \emph{structural} consequence relations, viz. consequence relations that are invariant under substitution, it is not sufficient to study an isolated suplattice. Inspired by how we have given a functorial treatment of syntax with built in substitutional structure, we extend our framework to further study consequence relations based on a pure functorial setting. Explicitly, we now consider functors of the form
\[ \mc A : \Set \to \SupL. \]
For any set map $f : X \to Y$, $\mc A$ induces a suplattice morphism
\[ \mc Af : \mc AX \to \mc AY. \]
This functorial dependence are used to model the abstract \emph{point-wise substitution}, as we will see more clearly later when we discuss concrete examples.

Let us first observe one very important functor of this type. As mentioned before, suplattices actually have a very close connection with the covariant power set functor $\mb P$.\footnote{The power set construction can be extended to a functor in both covariant and contravariant ways. The former is an endo-functor on $\Set$, while the latter is a functor from $\Set\op$ to $\Set$. In this paper we will exclusively consider the covariant case.} This naturally relates to our functorial approach to the usual examples of consequence relations in concrete logical systems, since as we've already seen the usual consequence relations are almost always defined on the power set of some sets related to the set of formulas.

The power set functor $\mb P$ sends each set $X$ to its power set $\mb PX$, and every function $f : X \to Y$ to a function $\mb P f : \mb P X \to \mb P Y$, which takes each subset $S$ of $X$ to the image $f(S)$ in $Y$. The crucial point is that $\mb P$ also has a monad structure $(\mb P,\eta_\mb P,\mu_\mb P)$. For any set $X$, the unit $\eta_{\mb P,X}$ sends $x$ to the singleton set $\set x$; the multiplication $\mu_{\mb P,X}$ sends a set of subsets of $X$ to their union. Just as Lemma~\ref{lem:fsigalg} shows for the syntactic monad $\mb F$ that $\mb F$-algebras are the same as $\Sigma$-algebras, it is well-known that $\mb P$-algebras are exactly \emph{suplattices}, i.e. we have an equivalence --- isomorphism, actually --- of categories
\[ \SupL \cong \Set^\mb P. \]
By the general theory of monads, this in particular shows that there is an induced adjunction
\[ \ms P : \Set \gb \SupL : U_\mb P. \]
For each set $X$, $\ms PX$ gives the free suplattice on $X$, which is the usual suplattice structure on the carrier set $\mb PX$ with inclusion being the partial order. In other words, the power set $\mb PX$ as a suplattice is free on the generating set $X$. This free suplattice functor $\ms P$ is then a typical example of a functor from $\Set$ to $\SupL$, and it is the free one. This makes it very convenient for us to describe consequence relations on concrete formulas in our functorial approach --- another reason to work with suplattices, rather than complete lattices.

Given any such functor $\mc A:\Set\to\SupL$, we also use $A$ to denote the composition with the forgetful functor from $\SupL$ to $\Set$,
\[ A = \mc A \cps U_\mb P : \Set \to \Set \]
For any set $X$, $AX$ considered as a set then has a unique $\mb P$-algebra structure induced by $\mc A$; hence we will also loosely view $AX$ as a suplattice when needed so. For any set map $f : X \to Y$, the induced map $Af$ is then a $\mb P$-algebra morphism, which means it preserves arbitrary joins in $AX$. We use $f^*(x)$ to denote $Af(x)$ for any $x\in AX$ when there is no confusion; of course, in that case $f_*$ will denote its right adjoint.

All such functors form a category $[\Set,\SupL]$. A morphism in $[\Set,\SupL]$ between two such functors would then a natural transformation
\[ q : \mc A \to \mc B. \]
This explicitly means that for any set $X$, $q_X$ is a morphism between suplattices; and for any set map $f : X \to Y$ the following naturality diagramme commutes
\[
\begin{tikzcd}
  \mc AX \ar[r, "q_X"] \ar[d, "\mc Af"'] & \mc BX \ar[d, "\mc Bf"] \\
  \mc AY \ar[r, "q_Y"'] & \mc BY
\end{tikzcd}
\]
On the level of their underlying sets, we also write
\[ q_X^* : AX \gb BX : q_{X,*}, \]
such that $q_X^*$ basically denotes the underlying function of $q_X$, and $q_{X,*}$ is the right adjoint of $q_X^*$.

One thing to notice is that for any such functor $\mc A$, $A$ would be an \emph{internal poset} in $[\Set,\Set]$. For any category $\catC$, an internal binary relation $R$ on some object $C$ in $\catC$ is simply a subobject $R \hook C \times C$. There are also definitions for an internal binary relation to be reflexive, symmetric, asymmetric, transitive, etc.. For our purposes though, it is enough to note that in a functor category like $[\Set,\Set]$, an internal binary relation is reflexive, symmetric, asymmetric, or transitive, if and only it is so component-wise. In particular, a binary relation $R \hook C \times C$ in $[\Set,\Set]$ is a preorder (resp. partial order) if and only if for any set $X$, $RX$ is a preorder (resp. partial order) on $CX$.

Now given a functor $\mc A : \Set \to \SupL$, the induced functor $A$ is indeed an internal poset in $[\Set,\Set]$. We have a following binary relation
\[ \operatorname\preceq \hook A \times A, \]
where for each set $X$
\[ \preceq_X = \scomp{(x,y)}{x \le y \text{ in } \mc AX}. \]
It is indeed a subfunctor of $A \times A$, because for any set map $f : X \to Y$, $Af$ is a map between suplattices, hence in particular monotone, which means that
\[ x \le y \nt f^*(x) \le f^*(y). \]
Evidently, $\preceq$ is an internal partial order in $[\Set,\Set]$. We also use $\succeq$ to denote its dual order; explicitly, for any set $X$ we have
\[ \succeq_X = \scomp{(x,y)}{ x \ge y \text{ in } \mc AX}. \]

To further address the point that for a functor $\mc A : \Set \to \SupL$ the induced suplattice morphisms $\mc Af$ should be understood as abstract point-wise substitution, we look at the free suplattice functor associated to the power set monad for example. We show how the functor $\ms P$ models the more complicated substitution on the level of power set of the set of formulas.

Consider the following composite functor
\[ \ms P\mb F : \Set \to \SupL. \]
It takes a set $X$ to the suplattice $\ms P\mb FX$ of the power set of the set of formulas $\mb FX$. Given any set map $f : X \to Y$ considered as specifying the way of substituting variables, there is an induced adjunction
\[ f^* : \mb P\mb FX \gb \mb P\mb FY : f_*. \]
By definition, for any $\Phi \subseteq \mb FX$ we have
\[ f^*(\Phi) = \mb P\mb Ff(\Phi) = \scomp{\varphi^f}{\varphi \in \Phi} = \Phi^f. \]
This means that $\ms P\mb F$ acting on a set map $f$ models point-wise substitution, which is exactly what we intend it to be. And simply from definition, such an operation preserves unions. Also, the left adjoint models inverse substitution: For any $\Psi \subseteq \mb FY$, by the description of right adjoint we have stated before,
\[ f_*(\Psi) = \scomp{\Phi\subseteq\mb FX}{\Phi^f \subseteq \Psi}. \]

Of course, the functor $\ms P$ itself is also another functor from $\Set$ to $\SupL$. We can view it as modelling the point-wise substitution when we have a \emph{trivial} language, i.e. there are no logic connectives in the language at all. The syntactic monad in this case is then simply given by the identity functor on $\Set$. The functor $\ms P$ then models point-wise substitution on the power set of formulas of this trivial language.

Our main goal in this section is to prove a similar result of Lemma~\ref{lem:bijecpoint}, showing that in the functorial setting the quotients are again in bijective correspondence to consequence relations, which lifts our general philosophy of viewing logics as a general way of constructing quotients to this functorial setting. To show this we first describes what quotients and structural consequence relations are in this new context.

A quotient of $\mc A$ should be understood as a surjection in the category $[\Set,\SupL]$, but again only considered up to isomorphism. Since limits and colimits are computed component-wise in a functor category, a surjection is then a natural transformation
\[ q : \mc A \surj \mc B, \]
such that for every set $X$ the component $q_X : \mc AX \surj \mc BX$ is a surjection of suplattices. In other words, $q$ is a quotient in the functor category $[\Set,\SupL]$ if and only if it is so for every component. Similar results holds for injections as well.

In the same spirit, we would also want to define what consequence relations are in this functorial setting in a component-wise manner. A structural consequence relation $\cons$ on $\mc A$ is a subfunctor on $A \times A$
\[ \cons \hook A \times A, \]
such that, point-wise, $\cons_X$ is a consequence relation on the suplattice $AX$ for any set $X$. In particular, this implies that $\cons$ is an \emph{internal preorder} on $A$ that contains $\succeq$,
\[
\begin{tikzcd}[column sep=0.6cm]
  \succeq \ar[rr, dashed, hook] \ar[dr, tail] & & \cons \ar[dl, tail] \\
  & A \times A &
\end{tikzcd}
\]
But a structural consequence relation $\cons$ in this functoriality setting contains more information than the mere point-wise consequence relations $\cons_X$. Functoriality of $\cons$ further implies \emph{structurality}, i.e. these point-wise consequence relations are invariant under substitutions. Previously we have shown that the functor $\mc A$, or similarly $A$, acting on morphisms models point-wise substitution; $\cons$ being a subfunctor of $A \times A$ means precisely that for any $x,y\in AX$, we have
\[ x \cons_X y \nt f^*(x) \cons_X f^*(y). \]
In our concrete example for consequence relations on $\ms P\mb F$, this means that for any subsets $\Phi,\Psi \subseteq \mb FX$ we have
\[ \Phi \cons_X \Psi \nt \Phi^f \cons_X \Psi^f, \]
exactly saying that the consequence relation is structural.

After given the necessary definitions, we first show a lemma stating that given a structural consequence relation, closed theories in each suplattice $AX$ are preserved by inverse substitution:
\begin{lemma}
\label{lem:closeinversesubst}
  Let $\cons$ be a structural consequence relation on $\mc A$. Then for any set map $f : X \to Y$, if $y\in AY$ is closed under $\cons_Y$, i.e. we have
  \[ y = \bigvee_{y \cons_Y w} w, \]
  then so is $f_*(y)\in AX$ under $\cons_X$.
\end{lemma}
\begin{proof}
  By definition, to prove $f_*(y)$ is closed we only need to show
  \[ f_*(y) \cons_X z \nt z \le f_*(y). \]
  Note the following calculation
  \begin{align*}
    f_*(y) \cons_X z
    &\nt f^*f_*(y) \cons_Y f^*(z), \\
    &\nt y \cons_Y f^*(z), \\
    &\nt f^*(z) \le y, \\
    &\nt z \le f_*(y).
  \end{align*}
  The first implication holds by functoriality of $\cons$; the second holds because $f^*f_*(y) \le y$ thus $y \cons_Y f^*f_*(y)$, and transitivity of $\cons_Y$ implies the remainder; the third holds since $y$ is closed; and the final implication holds because $f_*$ is right adjoint to $f^*$.
\end{proof}

We can now show in the functorial setting that structural consequence relations are again the same as quotients.
\begin{proposition}
\label{prop:funcbij}
  For any $\mc A$ in $[\Set,\SupL]$, consequence relations and quotients on $\mc A$ corresponds bijectively to each other.
\end{proposition}
\begin{proof}
  Given a quotient $q : \mc A \surj \mc B$, for any set $X$ we have a quotient map between suplattices
  \[ q_X^* : AX \surj BX, \]
  which by Lemma~\ref{lem:bijecpoint} uniquely induces a consequence relation $\cons_X$ on $AX$. Hence, we only need to show that the component-wise data $\cons_X \hook AX \times AX$ we get for every set $X$ indeed organise themselves into a subfunctor. Explicitly, we need to show that for any set map $f : X \to Y$ and any $x,y\in AX$,
  \[ x \cons_X y \nt f^*(x) \cons_Y f^*(y). \]
  Recall from Lemma~\ref{lem:bijecpoint}, the induced consequence relation is defined as follows
  \[ x \cons_X y \eff q_X^*(y) \le q_X^*(x). \]
  We then have
  \begin{align*}
    x \cons_X y
    &\nt q_X^*(y) \le q_X^*(x), \\
    &\nt f^*q_X^*(y) \le f^*q_X^*(x), \\
    &\nt q_Y^*f^*(y) \le q_Y^*f^*(x), \\
    &\nt f^*(x) \cons_Y f^*(y).
  \end{align*}
  The first and fourth implication holds by the definition of the induced consequence relation; the second implication holds by the fact that $f^*$ is monotone; the third holds by naturality of $q$. This then close one direction of the proof.

  On the other hand, suppose we are given a structural consequence relation on $\mc A$
  \[ \cons \hook A \times A. \]
  For any set $X$, $\cons_X$ is a consequence relation on $AX$. By Lemma~\ref{lem:bijecpoint} again, it uniquely induces a quotient map
  \[ \gamma_X^* : AX \surj A_\gamma X, \]
  where $A_\gamma X$ is the following set,
  \[ A_\gamma X = \scomp{x\in AX}{x = \bigvee_{x\cons_X y} y}. \]
  $\gamma_X$ sends any $x$ in $AX$ to its closure $\bigvee_{x\cons_X y} y$. Similarly, we only need to show the functoriality of this construction. Given any set map $f : X \to Y$, by Lemma~\ref{lem:closeinversesubst} we know that the right adjoint $f_*$ restricts to a map
  \[ f_{\gamma,*} : A_\gamma Y \to A_\gamma X. \]
  This restricted map actually preserves arbitrary meets, since, as we've mentioned before, meets in $A_\gamma X$ and $A_\gamma Y$ are computed the same as in $AX,AY$, respectively, and $f_*$ as a right adjoint preserves arbitrary meets. This in particular implies that we have a left adjoint
  \[ f_\gamma^* : A_\gamma X \to A_\gamma Y. \]
  Explicitly, the left adjoint $f_\gamma^*$ is given by
  \[ f_\gamma^*(x) = \gamma_Yf^*(x), \]
  where $\gamma_Y$ is the induced closure operator on $AY$. We show this by observing that for any $x \in AX$ and any $y \in A_\gamma Y$, since $y$ is closed the following holds,
  \[ \gamma_Yf^*(x) \le y \eff f^*(x) \le y \eff x \le f_*(y). \]
  This implies that we have a well-defined functor $\mc A_\gamma$. To now prove $\gamma$ is a quotient in $[\Set,\SupL]$, we only need to show the following naturality diagramme commutes,
  \[
  \begin{tikzcd}
    AX \ar[r, two heads, "\gamma_X^*"] \ar[d, "f^*"'] & A_\gamma X \ar[d, "f_\gamma^*"] \\
    AY \ar[r, two heads, "\gamma_Y^*"'] & A_\gamma Y
  \end{tikzcd}
  \]
  Explicitly, we need to show that for any $x\in AX$, the closure of $f^*\gamma_X(x)$ and $f^*(x)$ coincide. Since we know that $f^*(x) \le f^*\gamma_X(x)$, we only need to show
  \[ f^*(x) \cons_Y f^*\gamma_X(x). \]
  However, since every element implies its closure
  \[ x \cons_X \gamma_X(x), \]
  functoriality of $\cons_X$ then proves the above fact. Hence, we have shown that such a consequence relation indeed induces a quotient
  \[ \gamma : \mc A \surj \mc A_\gamma. \]

  Finally, we need to prove the two constructions are mutually inverse to each other. One direction is easier. Given a consequence relation $\cons$ on $\mc A$, since both $A_\gamma X$ and $\cons_{\gamma,X}$ are induced component-wise for any set $X$, it is a direct consequence of Lemma~\ref{lem:bijecpoint} that $\cons$ and $\cons_\gamma$ are the same consequence relation. On the other hand, given a quotient $q : \mc A \surj \mc B$, we need to show that the quotient $\gamma : \mc A \to \mc A_\gamma$ induced by the consequence relation $\cons_q$ is isomorphic to $q$. Again by the proof of Lemma~\ref{lem:bijecpoint}, component-wise we have an isomorphism
  \[ q_{X,*} : BX \cong A_\gamma X, \]
  for any set $X$. It is inherited from the right adjoint $q_{X,*}$ from $BX$ to $AX$. It remains to show naturality of this isomorphism, viz. to prove the following diagramme commutes,
  \[
  \begin{tikzcd}
    BX \ar[r, "q_{X,*}"] \ar[d, "Bf"'] & A_\gamma X \ar[d, "f_\gamma^*"] \\
    BY \ar[r, "q_{Y,*}"'] & A_\gamma Y
  \end{tikzcd}
  \]
  For any $x\in BX$, by definition we have
  \begin{align*}
    f_\gamma^*(q_{X,*}(x))
    &= \gamma_YAf\bigvee_{q^*_X(y) = x} y, \\
    &= q_{Y,*}\bigvee_{q^*_X(y) = x} q_Y^*(Af(y)), \\
    &= q_{Y,*} \bigvee_{q^*_X(y) = x} Bf(q_X^*(y)), \\
    &= q_{Y,*}(Bf(x)).
  \end{align*}
  The first equality holds by the definition of the right adjoint $q_{X,*}$ and $f_\gamma^*$; the second holds since $\gamma_Y = q_Y^* \cps q_{Y,*}$ and both $q_Y^*$ and $Af$ preserves joins; the third holds by naturality of $q$; and the final equality holds by direction computation. As a result,
  \[ q_{X,*} \cps f_\gamma^* = Bf \cps q_{Y,*}. \]
  Hence, the component-wise isomorphisms between $BX$ and $A_\gamma X$ indeed organize themselves to a natural isomorphism $\mc B \cong \mc A_\gamma$, representing the same quotient.
\end{proof}
We end this section by providing some abstract examples of how to build further consequence relations on existing ones. In the next section, we will initiate a more concrete study of how a general class of consequence relations can be induced from a class of algebras in our functorial setting, which closely connects to the usual algebraic semantics of a logic.
\begin{example}
\label{exp:openquotient}
  Suppose now we have a consequence relation $\cons$ on $\mc A$. We show how a generalised element of $A$ would induce a new structural consequence relation based on $\cons$. Let $a$ be a generalised element of $A$, viz. a natural transformation
  \[ a : p \to A. \]
  For any set $X$ we define a new consequence relation $\cons_X^a$ on $AX$ as follows: For any $x,y\in AX$, we define
  \[ x \cons_X^a y \eff x \vee \bigvee_{i \in p(X)}a_X(i) \cons_X y. \]
  The induced family of consequence relations $\cons_X^a$ is still structural. Given any $f : X \to Y$, we observe
  \begin{align*}
    f^*\left(x \vee \bigvee_{i \in p(X)}a_X(i) \right)
    &= f^*(x) \vee \bigvee_{i \in p(X)}f^*(a_X(i)), \\
    &= f^*(x) \vee \bigvee_{i \in p(X)}a_Y(pf(i)), \\
    &\le f^*(x) \vee \bigvee_{j \in p(Y)}a_Y(j).
  \end{align*}
  This in particular implies that
  \[ f^*(x) \vee \bigvee_{j \in p(Y)}a_Y(j) \cons_Y f^*\left(x \vee \bigvee_{i \in p(X)}a_X(i) \right). \]
  By structurality of $\cons$ and the above fact, we then have
  \[ x \cons_X^a y \nt f^*(x) \cons_Y^a f^*(y). \]
  Intuitively, the newly constructed structural consequence relation $\cons^a$ is one that induced from $\cons$ by always considering those elements in $AX$ from $p(X)$ along $a$ as axioms.

  The majority of cases that will be interesting in ordinary logical studies is when $p$ is a representable functor $\yon^X$. By the Yoneda lemma, a morphism $a$ from $\yon^X$ to $A$ is the same as an element in $a \in A(X)$. Given such an element, by definition we have
  \[ x \cons_y^a y \eff x \vee \bigvee_{\sigma : X \to Y} a^\sigma \cons_X y. \]
  This corresponds to adding the axiom scheme represented by $a \in A(X)$, viz. a collection of axioms closed under substitution, into our logical system.
\end{example}

Proposition~\ref{prop:funcbij} makes it clear that we can identity a consequence relation on a functor $\mc A$ as a quotient of $\mc A$, providing a purely categorical description of consequence relations. We will then use this identification in subsequent texts to study algebraisation of logics and semantics in general. These will be the topics of further sections.

\section{Algebraically Induced Consequence Relations}
\label{sec:algseman}
In this section, we consider abstract consequence relations generated by a subclass of algebras, which further constitutes the algebraic semantics of our logic. However, the general account of semantics of logic will be described in more detail in later sections. Let's first observe how the usual notion of an algebraic model of a logic on a fixed set of variables can be presented in our description of syntax using the syntactic monad $\mb F$.

Recall that the syntactic monad $\mb F$ induces an adjunction
\[ \ms F : \Set \gb \Set^\mb F : U_\mb F. \]
The free $\mb F$-algebra functor $\ms F$ sends any set $X$ to the \emph{free} $\mb F$-algebra $\mu_{\mb F,X} : \mb F\mb FX \to \mb FX$ on the generating set $X$.\footnote{Recall that $\mb F$-algebras and $\Sigma$-algebras are essentially the same thing.} Explicitly, it means that given any $\mb F$-algebra $\alpha : \mb F\mbb S \to \mbb S$ and any set map $e : X \to \mbb S$, there is a \emph{uniquely} induced $\mb F$-algebra morphism
\[ \qsi e : \mb FX \to \mbb S, \]
which means the following diagramme commutes,
\[
\begin{tikzcd}
  \mb F\mb FX \ar[r, "\mu_{\mb F,X}"] \ar[d, "\mb F\qsi e"'] & \mb FX \ar[d, "\qsi e"] \\
  \mb F\mbb S \ar[r, "\alpha"'] & \mbb S
\end{tikzcd}
\]
Explicitly, the function $\qsi e$ is given by the following composite
\[ \qsi e = \mb F e \cps \alpha. \]
By naturality of $\mu_\mb F$ and the fact that $\mbb S$ with $\alpha$ is an $\mb F$-algebra we can indeed show the above diagramme commutes. In the usual context, $X$ is understood as the set of variables, and the function $e$ is simply an evaluation function of atoms. The uniquely induced $\mb F$-algebra homomorphism $\qsi e$ is usually inductively defined manually based on $e$. This constitutes an algebraic model of $\mb FX$.
\begin{definition}
  An \emph{algebraic model} of $\mb FX$ is a pair $\pair{\mbb S,e}$, where $\mbb S$ is equipped with an $\mb F$-algebra structure $\alpha : \mb F\mbb S \to \mbb S$ and $e$ is an evaluation function from $X$ to $\mbb S$. Given an algebraic model $\pair{\mbb S,e}$, for any formula $\varphi \in \mb FX$, we write
  \[ \ass \varphi^e = \qsi e(\varphi), \]
  as denoting the evaluation of $\varphi$ in $\mbb S$ induced by $e$.
\end{definition}

In later sections we will study more generally what a genuine functorial description of semantics could be in our framework. Algebraic models in the above sense naturally provides the notion of satisfaction on \emph{pairs} of formulas. For any $\varphi,\psi\in\mb FX$, we will write $\varphi\approx\psi$ to denote the pair $(\varphi,\psi)$. Given an algebraic model $\pair{\mbb S,e}$, we say a pair $\varphi \approx \psi$ is \emph{satisfied} in $\pair{\mbb S,e}$, denoted as
\[ \mbb S, e \vDash_X \varphi \approx \psi, \]
if $\ass\varphi^e = \ass\psi^e$. For a set of pairs $E\subseteq\mb FX \times \mb FX$, we also write
\[ \mbb S, e \vDash_X E, \]
if for all pairs $\varphi\approx\psi\in E$, $\mbb S,e\vDash_X\varphi\approx\psi$.

The above description of algebraic models for $\mb FX$ over a set of variables $X$ then further leads us to study the consequence relations induced by a class of algebras. As hinted above, this consequence relation will not be based on $\ms P\mb F$, but should relate to \emph{pairs} of formulas. We use $\Eq$ to denote the functor $\mb F \times \mb F$. Recall at the end of Section~\ref{sec:synmonad} we have discussed how products of $\mb F$ behave: $\Eq$ sends each set $X$ to the product $\mb FX \times \mb FX$, the set of pairs of formulas over $X$. For any function $f : X \to Y$, the action of $\Eq$ on morphisms simply do substitution along $f$ for each entry of the pair
\[ (\varphi \approx \psi)^f = \Eq f(\varphi\approx\psi) = \varphi^f \approx \psi^f. \]
Similarly, for a subset $E\subseteq\Eq X$, we also write
\[ E^f = \mb P\Eq f (E) = \scomp{\varphi^f \approx \psi^f}{\varphi\approx\psi \in E}. \]

This point-wise definition of satisfaction naturally induce the notion of validity for a given a subclass $\mc K$ of $\mb F$-algebras, or more precisely a full subcategory of $\Set^\mb F$. We describe now how it will induce a consequence relation on the functor $\ms P\Eq$. For any $E,F\subseteq\Eq X$, we say
\[ E \vDash_{\mc K,X} F, \]
if for any algebraic model $\pair{\mbb S,e}$ in $\mc K$, $\mbb S,e \vDash_X E$ implies $\mbb S,e \vDash_X F$. As usually in logic, we if $E$ is the empty set we simply write
\[ \vDash_{\mc K,X} F. \]
Just like in the usual formulation of algebraic semantics of logic, $\vDash_{\mc K,X}$ is a component-wise consequence relation on the suplattice $\ms P\Eq X$, for any set $X$. To see that such data assemble themselves to a \emph{structural} consequence relation for the \emph{functor} $\ms P\Eq$, we need to further show the functoriality of $\vDash_{\mc K,X}$.

To do this, we first give a useful fact.
\begin{fact}
\label{lem:funcass}
  Let $\mc K$ be a subclass of $\mb F$-algebras. Given any $\mb F$-algebra $\alpha : \mb F\mbb S \to \mbb S$ and any evaluation map $e : X \to \mbb S$, for any function $f : Y \to X$ and for any $\varphi \in \mb FY$ the following holds
  \[ \ass{\varphi^f}^e = \ass{\varphi}^{f\cps e}. \]
\end{fact}
\begin{proof}
  This can be directly calculated from relevant definition on evaluation function of formulas and the functorial representation of substitution: For any $\varphi\in\mb FY$ we have
  \[ \ass{\varphi^f}^e = \qsi e(\mb Ff\varphi) = \alpha(\mb Fe(\mb Ff\varphi)) = \alpha\mb F(f\cps e)(\varphi) = \ass\varphi^{f \cps e}. \qedhere \]
\end{proof}
We can now show that the class of algebras $\mc K$ induces a well-defined structural consequence relation $\vDash_{\mc K}$ on $\ms P\Eq$, such that for any set $X$, its component $\vDash_{\mc K,X}$ is the previously defined component-wise consequence relation on $\ms P\mb FX$.
\begin{proposition}
\label{prop:algecons}
  The component-wise relations $\vDash_{\mc K,X}$ extends to a subfunctor
  \[ \operatorname{\vDash_\mc K} \hook \mb P\Eq \times \mb P\Eq, \]
  constituting a structural consequence relation on $\ms P\Eq$.
\end{proposition}
\begin{proof}
  As mentioned before, we only need to show the functoriality of $\vDash_{\mc K}$, i.e. given any function $f : Y \to X$ and $E,F\subseteq\Eq Y$, we need to show that
  \[ E \vDash_{\mc K,Y} F \nt E^f \vDash_{\mc K,X} F^f. \]
  For any $\mb F$-algebra $\alpha : \mb F\mbb S \to \mbb S$ in $\mc K$, any evaluation function $e : X \to \mbb S$ and any function $f : Y \to X$, from Fact~\ref{lem:funcass} we have the following observation: For any subset $E \subseteq \Eq Y$, we have
  \begin{align*}
    \mbb S,e \vDash_{X} E^f
    &\eff \forall \varphi\approx\psi \in E,\ \mbb S,e \vDash_{X} \varphi^f \approx \psi^f, \\
    &\eff \forall \varphi \approx \psi \in E,\ \ass{\varphi^f}^e = \ass{\psi^f}^e, \\
    &\eff \forall \varphi \approx \psi \in E,\ \ass{\varphi}^{f \cps e} = \ass{\psi}^{f \cps e}, \\
    &\eff \forall \varphi \approx \psi \in E,\ \mbb S, f\cps e \vDash_{X} \varphi \approx \psi, \\
    &\eff \mbb S,f\cps e \vDash_{Y} E.
  \end{align*}
  Now suppose we have $E,F \subseteq \Eq Y$ that $E \vDash_{\mc K,Y} F$. Then by the above computation we have
  \[ \mbb S,e \vDash_{X} E^f \nt \mbb S,f\cps e \vDash_{Y} E \nt \mbb S,f\cps e\vDash_{Y} F \nt \mbb S,e \vDash_{X} F^f. \]
  The first and last implication follows from the above equivalence; the second implication follows from the fact that $E \vDash_{\mc K,Y} F$. This completes the proof that $\vDash_\mc K$ is functorial, and thus constitutes a structural consequence relation on $\ms P\Eq$.
\end{proof}

We end this section by considering the special case where $\mc K$ is a variety of algebras. A variety of algebras, by definition, is an equational class; it is a subclass of the class of all algebraic structures of a given signature that satisfies a given set of identities. In particular, it is (strictly) monadic over $\Set$: There exists a monad $\mb K$, such that we have an equivalence (isomorphism) of categories
\[ \mc K \cong \Set^\mb K. \]
Such a setting allows us to generate the Lindenbaum–Tarski algebra construction from general category theory, by exploring the duality between monads and monad maps on one hand, and categories of monad algebras (monadic functors) on the other hand.

The collection of monads over $\Set$ actually forms a category $\Mon$. Given two monads $\mb F,\mb K$, a morphism from $\mb F$ to $\mb K$ in $\Mon$ is a natural transformation
\[ \lambda : \mb F \to \mb K, \]
that interacts well with the monad structure of both $\mb F$ and $\mb K$, making the following two diagrammes commute,
\[
\begin{tikzcd}
  \yon \ar[r, "\eta_\mb F"] \ar[dr, "\eta_\mb K"'] & \mb F \ar[d, "\lambda"] \\
  & \mb K
\end{tikzcd} \quad \quad
\begin{tikzcd}
  \mb F \mb F \ar[r, "\mu_\mb F"] \ar[d, "\lambda \circ \lambda"'] & \mb F \ar[d, "\lambda"] \\
  \mb K \mb K \ar[r, "\mu_\mb K"'] & \mb K
\end{tikzcd}
\]
Such a monad map induces a pullback functor between the corresponding categories of algebras
\[ \lambda^* : \Set^\mb K \to \Set^\mb F. \]
For any $\mb K$-algebra $\alpha : \mb KS \to S$, $\lambda^*\alpha$ is the following composition
\[ \lambda^*\alpha = \lambda_S \cps \alpha : \mb FS \to \mb KS \to S. \]
One can directly verify that $\lambda^*\alpha$ is indeed an $\mb F$-algebra, using the fact that $\lambda$ is a monad map and $\alpha$ is a $\mb K$-algebra, which we leave for the readers to check. For a given $\mb K$-algebra morphism $f : S \to T$ from $\alpha : \mb KS \to S$ to $\beta : \mb KT \to T$, it also lifts to a morphism $\lambda^*f$ from $\lambda^*\alpha$ to $\lambda^*\beta$ by simply using $f$ itself,
\[
\begin{tikzcd}
  \mb FS \ar[r, "\lambda_S"] \ar[d, "\mb Ff"] & \mb KS \ar[r, "\alpha"] \ar[d, "\mb Kf"] & S \ar[d, "f"] \\
  \mb FT \ar[r, "\lambda_T"'] & \mb KT \ar[r, "\beta"'] & T
\end{tikzcd}
\]
It is easy to see from the above definition that the induced functor $\lambda^*$ between the two category of algebras are compatible with the forgetful functors $U_\mb K,U_\mb F$, i.e. we have
\[ \lambda^* \cps U_\mb F = U_\mb K. \]

On the other hand of the duality lies monadic functors on $\Set$ over $U_\mb F$. A monadic functor is by definition equivalent to a forgetful functor for some monads, and a morphism between monadic functors $U_\mb T,U_\mb F$ is then a single functor $G : \Set^\mb T \to \Set^\mb F$ that commutes with the forgetful functors $U_\mb T,U_\mb F$,
\[
\begin{tikzcd}
  \Set^\mb T \ar[dr, "U_\mb T"'] \ar[rr, "G"] & & \Set^\mb F \ar[dl, "U_\mb F"] \\
  & \Set &
\end{tikzcd}
\]
We use $\MnF$ to denote the category of monadic functors on $\Set$.

Recall that for any set $X$, $\ms KX$, the free $\mb K$-algebra on $X$, is given by the multiplication $\mu_{\mb T,X} : \mb T\mb T X \to \mb TX$. Since $G$ commutes with the forgetful functors, it then follows that the functor $G$ then associate an $\mb F$-algebra structure on $\mb TX$
\[ G\ms KX : \mb F\mb TX \to \mb TX. \]
Again, $\mu_{\mb F,X}$ is the free $\mb F$-algebra on $X$, and we have a canonical map $\eta_{\mb T,X}$. This then implies that we have a canonically induced $\mb F$-algebra morphism
\[ \qsi{\eta_{\mb T,X}} = \mb F\eta_{\mb T,X} \cps G\ms KX : \mb FX \to \mb F\mb TX \to \mb TX. \]
We can then simply define a natural transformation $\delta : \mb F \to \mb T$, making the $X$-component of $\delta$ be $\qsi{\eta_{\mb T,X}}$. Naturality of $\delta$, and the fact that it is furthermore a monad map from $\mb F$ to $\mb T$, follows from the fact that $G$ is a functor that commutes with the two forgetful functors. It is well-known that the above described two-sided constructions indeed form a functorial bijective correspondence. We refer the readers to \cite[p.~108]{manes2003monads} for a proof. Here we in particular notice the following refined correspondence:
\begin{lemma}
\label{lem:quomonadfullsub}
  Quotient monad maps $\lambda : \mb F \surj \mb K$ on syntactic monad $\mb F$, viz. monad maps that are component-wise surjective out of $\mb F$, bijectively corresponds to varieties of algebras $\mc K$ that form a full subcategory of $\Set^\mb F$.
\end{lemma}
\begin{proof}[Proof Sketch]
  We first show monad maps $\lambda : \mb F \surj \mb K$ that are component-wise surjective induces a full subcategory inclusion $\Set^\mb K \hook \Set^\mb F$, or more precisely, the inclusion functor is fully faithful and injective on objects. Given a $\mb K$-algebra $\alpha : \mb KX \to X$, according to the previously mentioned construction, the induced $\mb F$-algebra is given by
  \[
  \begin{tikzcd}
    \mb F X \ar[r, two heads, "\lambda_X"] & \mb KX \ar[r, "\alpha"] & X.
  \end{tikzcd}
  \]
  Since $\lambda_X$ is surjective, it is obvious that there are no $\mb K$-algebra structure on $X$ that gives out the same $\mb F$-algebra. Thus, the inclusion is injective on objects. It is easy to see that the inclusion is faithful. We only need to show it is full, i.e. in the following diagramme, if the outer square commutes then so does the right square
  \[
  \begin{tikzcd}
    \mb FX \ar[r, two heads, "\lambda_X"] \ar[d, "\mb Ff"] & \mb KX \ar[d, "\mb Kf"] \ar[r, "\alpha"] & X \ar[d, "f"] \\
    \mb FY \ar[r, "\lambda_Y"', two heads] & \mb KY \ar[r, "\beta"'] & Y
  \end{tikzcd}
  \]
  Notice that the left square commutes by naturality of $\lambda$. We finally observe that
  \[
  \lambda_X \cps \alpha \cps f = \mb Ff \cps \lambda_Y \cps \beta = \lambda_X \cps \mb Kf \cps \beta.
  \]
  This equation, together with the fact that $\lambda_X$ is surjective, shows that
  \[ \alpha \cps f = \mb Kf \cps \beta. \]
  Hence, the inclusion is also full.

  The other direction of the proof is more tricky. It will use the fact that our syntactic monad preserves surjections, and the fact that a variety of full subcategory of $\Set^\mb F$ is closed under forming products, subalgebras, and reflexive coequalisers, which ultimately relies on Birkhoff's theorem. Or we can use an adjoint lifting theorem described in \cite{johnstone1975adjointlifting}, showing that $\mc K$ would be a reflexive subcategory of $\Set^\mb F$. We refer the readers to \cite[p.~110]{manes2003monads} for a complete proof.
\end{proof}

We further show that the induced quotient map on monads is essentially providing us with the Lindenbaum–Tarski algebra construction, by observing the following fact:
\begin{proposition}
\label{prop:ltalge}
  For any set $X$ and any pair of formulas $\varphi \approx \psi \in \Eq X$,
  \[ \vDash_{\mc K,X} \varphi \approx \psi \eff \mb KX,\eta_{\mb K,X} \vDash_X \varphi \approx \psi, \]
  where $\mb KX$ is viewed as the free $\mb K$-algebra on $X$ included in $\Set^\mb F$.
\end{proposition}
\begin{proof}
  The left to right direction is trivial. For the right to left direction, we show that $\lambda_X$ consists of an $\mb F$-algebra morphism from $\ms KX$ included into $\Set^\mb F$, to $\ms FX$; and it is the initial one among $\mc K$, i.e. $\mb F$-algebras that lies in the image of $\lambda^*$.

  We first observe that the following diagramme commutes by the fact that $\lambda$ commutes with multiplications of $\mb F$ and $\mb K$,
  \[
  \begin{tikzcd}[row sep = 1cm]
    \mb F \mb K X \ar[r, "\lambda_{\mb KX}"] & \mb K\mb KX \ar[r, "\mu_{\mb K,X}"] & \mb KX \\
    \mb F\mb F X \ar[u, "\mb F\lambda_{X}"] \ar[rr, "\mu_{\mb F, X}"'] \ar[ur, "(\lambda\circ\lambda)_X"'] & & \mb F X \ar[u, "\lambda_X"']
  \end{tikzcd}
  \]
  We also know
  \[ \eta_{\mb F,X} \cps \lambda_X = \eta_{\mb K,X}. \]
  Together they show that $\lambda_X$ is the uniquely induced $\mb F$-algebra morphism from the set map $\eta_{\mb K,X}$, since $\mb FX$ is free. We can also see this fact by direct computation,
  \[ \qsi{\eta_{\mb K,X}} = \mb F\eta_{\mb K,X} \cps \lambda_{\mb KX} \cps \mu_{\mb K,X} = \lambda_{X} \cps \eta_{\mb K,\mb KX} \cps \mu_{\mb K,X} = \lambda_X. \]

  Next we show the initialness of $\lambda_X$ of $\mb F$-algebra morphisms from $\mb FX$ to ones in $\mc K$. Explicitly, we show that for any $\mb F$-algebra $\mbb S$ of the form
  \[
  \begin{tikzcd}
    \mb F \mbb S \ar[r, "\lambda_\mbb S"] & \mb K\mbb S \ar[r, "\alpha"] & \mbb S
  \end{tikzcd},
  \]
  where $\alpha$ is a $\mb K$-algebra morphism, any $\mb F$-algebra homomorphism from $\mb FX$ to $\mbb S$ factors through $\lambda_X$:
  \[
  \begin{tikzcd}
    \mb F \mbb S \ar[r, "\lambda_\mbb S"] & \mb K\mbb S \ar[r, "\alpha"] & \mbb S \\
    \mb F \mb K X \ar[r, "\lambda_{\mb KX}"] \ar[u, "\mb F\ov e"] & \mb K\mb KX \ar[r, "\mu_{\mb K,X}"] \ar[u, "\mb K\ov e"] & \mb KX \ar[u, dashed, "\ov e"] \\
    \mb F\mb F X \ar[u, "\mb F\lambda_{X}"] \ar[rr, "\mu_{\mb F, X}"] & & \mb F X \ar[u, "\lambda_X"] \ar[uu, bend right, dashed, "\qsi e"']
  \end{tikzcd}
  \]
  For any set map $e : X \to \mbb S$ --- as previously mentioned, every $\mb F$-algebra homomorphism from $\ms FX$ to $\alpha$ is of the form $\qsi e$ for some $e$ since $\ms FX$ is free --- $\ov e$ is the uniquely induced $\mb K$-algebra morphism from $\ms KX$ to $\alpha$. The diagramme commutes because $\ov e$ is a $\mb K$-algebra homomorphism and due to the naturality of $\lambda$. By uniqueness of the induced map $\qsi e$, it then follows that $\qsi e$ factors through $\lambda_X$. We can also prove it by direct computation:
  \[ \qsi e = \mb Fe \cps \lambda_\mbb S \cps \alpha = \lambda_X \cps \mb Ke \cps \alpha = \lambda_X \cps \ov e. \]

  Finally, we suppose that
  \[ \mb KX,\eta_{\mb K,X} \vDash_X \varphi \approx \psi, \]
  or in other words
  \[ \lambda_X(\varphi) = \lambda_X(\psi). \]
  Then for any algebraic semantics $\pair{\mbb S,e}$ with $\mbb S$ an algebra in $\mc K$, the initialness of $\lambda_X$ implies
  \[ \qsi e(\varphi) = \ov e \circ \lambda_X(\varphi) = \ov e \circ \lambda_X(\psi) = \qsi e(\psi) \nt \mbb S,e \vDash_X \varphi \approx \psi. \]
  Thus, we have
  \[ \mb KX,\eta_{\mb K,X} \vDash_X \varphi \approx \psi \nt \operatorname\vDash_{\mc K,X} \varphi \approx \psi. \qedhere \]
\end{proof}

Proposition~\ref{prop:ltalge} then shows that the induced free $\mb K$-algebra $\ms KX$ with the evaluation $\eta_{\mb K,X}$ contains all the semantic information of the class $\mc K$ with respect to the set of formulas $\mb FX$. The evaluation map from $\mb FX$ to $\mb KX$ is precisely given by the $X$-component of the monad quotient map $\lambda$. Putting all these together then explicitly shows how the usual story of Lindenbaum-Tarski algebra can be recovered from our categorical setting. And by the previous Lemma~\ref{lem:quomonadfullsub}, given any such class of variety of algebras $\mc K$ serving as algebraic models and generating a structural consequence relation on $\ms P\Eq$, by general category theory there is a uniquely induced quotient monad $\lambda : \mb F \to \mb K$, that (1) on the functorial level it provides the construction of Lindenbaum-Tarski algebra, which is generally a \emph{quotient} of the syntactic term algebra; and (2) this quotient $\lambda_X$ on each component can then be viewed as evaluation map of our syntactic objects in $\mb FX$, which contains the full semantic information about the class $\mc K$. Such a setting is another incarnation of our general philosophy that logic is describing general ways of constructing quotients.

\section{Category of Consequence Relations}
\label{sec:equivconsrel}
In Section~\ref{sec:acrsup} we have identified structural consequence relations in its full generality as quotients of $\mc A$, with $\mc A$ being functors in $[\Set,\SupL]$. Section~\ref{sec:algseman} implements a concrete example of a consequence relation on the functor $\ms P\Eq$, with a given subclass of $\mb F$-algebras as providing semantics on pairs of formulas. However, when we study usual logics with an algebraic signature, as for propositional logic or other modal logics, the proof system we give does not generate consequence relations on $\ms P\Eq$, or $\mb P(\fml \times \fml)$ where the set of variables is fixed, but on $\ms P\mb F$ or $\mb P\fml$. However, viewing logics from a universal algebra perspective, especially the construction of Lindenbaum-Tarski algebras discussed in the previous section, has proven to be a very successful approach to study logical systems in general. And there is a precise notion of when such an algebraic study of logics is equivalent to the proof theoretic one. This is the notion of \emph{algebraisation} of logics, first systematically studied in \cite{blok1989algebraizable}.

Conceptually, a proof system, or a structural consequence relation on $\ms P\mb F$, is algebraisable if and only if it is equivalent --- in a sense we are going to state more precisely later --- to a consequence relation on $\ms P\Eq$ generated by a class of algebras as studied in the previous section. Hence, in the following texts we will first take up a general study of relations between different consequence relations. More specifically, we will investigate when a structural consequence relation is \emph{represented} or \emph{equivalent} to another one. In other words, we will be interested in the \emph{category of structural consequence relation}. Many of the results presented in this section is a direct generalisation of results obtained in \cite{galatos&tsinakis2009consequence}.

Now in Section~\ref{sec:acrsup}, Proposition~\ref{prop:funcbij} has provided us with a very nice identification, establishing that structural consequence relations corresponds bijectively to quotients in $[\Set,\SupL]$. We will heavily rely on such an identification in this section, since quotients are much more easy to describe categorically. We let $\Quo$ be the category of surjections in $[\Set,\SupL]$. Explicitly, objects in $\Quo$ are surjections $q : \mc A \surj \mc B$ in $[\Set,\SupL]$; a morphism from $q_1 : \mc A_1 \surj \mc B_1$ to $q_2 : \mc A_2 \surj \mc B_2$ is a pair of morphisms $(\tau,\rho)$ in $[\Set,\SupL]$ that makes the following diagramme commute,
\[
\begin{tikzcd}
  \mc A_1 \ar[r, "\tau"] \ar[d, two heads, "q_1"'] & \mc A_2 \ar[d, two heads, "q_2"] \\
  \mc B_1 \ar[r, "\rho"'] & \mc B_2
\end{tikzcd}
\]
A morphism in $\Quo$, in a sense, gives a way of translating the information of consequence relation $\cons$ induced by $q_1$ on $\mc A_1$ to the one $\conss$ induced by $q_2$ on $\mc A_2$. Suppose we are given such a morphism $(\tau,\rho)$, then for any set $X$ and any $x,y\in A_1X$, by definition we have
\begin{align*}
  x \cons_X y
  &\eff q_{1,X}^*(y) \le q_{1,X}^*(y), \\
  &\nt \rho_X^*(q_{1,X}^*(y)) \le \rho_X^*(q_{1,X}^*(x)), \\
  &\nt q_{2,X}^*(\tau_X^*(y)) \le q_{2,X}^*(\tau_X^*(x)), \\
  &\nt \tau_X^*(x) \conss_X \tau_X^*(y).
\end{align*}

Among all of such morphisms, we are in particularly interested in the case where it induces a \emph{faithful} translation. We say the consequence relation $\cons$ induced by $q_1$ is \emph{faithfully represented}, or simply \emph{represented}, by the consequence relation $\conss$ induced by $q_2$ along $(\tau,\rho)$, if for any set $X$ and any $x,y\in\mc AX$
\[ x \cons_X y \eff \tau^*_X(x) \conss_X \tau^*_X(y). \]
Below we give a precise characterisation of when a morphism in $\Quo$ induces a faithful representation.

Recall again that limits and colimits in a functor category like $[\Set,\SupL]$ are computed component-wise. In particular, what we have described as the image-factorisation of suplattices in Example~\ref{exp:imagefactor} can be almost seamlessly transported to the functorial case in $[\Set,\SupL]$. It turns out that representation of consequence relations has a close connection with image-factorisations.
\begin{lemma}
\label{lem:repinj}
  Let $(\tau,\rho)$ be a morphism in $\Quo$ from $q_1 : \mc A_1 \surj \mc B_1$ to $q_2 : \mc A_2 \surj \mc B_2$. Then it consists of a representation of $\cons$ by $\conss$ iff $q_1$ and $\rho$ make up of the image-factorisation of $\tau \cps q_2$, iff $\rho$ is injective.
\end{lemma}
\begin{proof}
  From the component-wise description of images in $[\Set,\SupL]$ we know that $\mc B_2$ is the image of $\tau\cps q_2$ iff $\rho$ is injective. Suppose we have such a morphism with $\rho$ being injective. Let $X$ be any set and $x,y$ any elements in $\mc A_1X$. From the correspondence between consequence relations and quotients, we have
  \[ x \cons_{X} y \eff q_{1,X}^*(y) \le q_{1,X}^*(x). \]
  Since $\rho$ is injective, in particular $\rho_X^*$ is injective for any set $X$, this means that $\rho_X^*$ is a poset-embedding of $\mc B_1X$ into $\mc B_2X$. As a result, for any $a,b\in\mc B_1X$, $a \le b$ if and only if $\rho_X^*(a) \le \rho_X^*(b)$. We can then further compute
  \begin{align*}
    x \cons_{X} y
    &\eff q_{1,X}^*(y) \le q_{1,X}^*(x), \\
    &\eff \rho^*_Xq_{1,X}^*(y) \le \rho^*_Xq_{1,X}^*(x), \\
    &\eff q_{2,X}^*\tau_X^*(y) \le q_{2,X}^*\tau_X^*(x), \\
    &\eff \tau_X^*(x) \conss_X \tau_X^*(y).
  \end{align*}
  The first and last equivalence follows from the correspondence between quotients and consequence relations; the second equivalence follows from the fact that each $\rho_X^*$ is a poset-embedding; the third equivalence is a result of the commuting diagramme in the definition of morphisms in $\Quo$.

  On the other hand, suppose $\cons$ is faithfully represented by $\conss$ in $(\tau,\rho)$. We need to show that $\rho$ is injective, which is equivalent to show that $\rho_X^*$ is injective for any set $X$. First, since $q_1$ is a surjection, it follows that every element in $\mc B_1X$ has the form $q_{1,X}^*(x)$ for some $x\in \mc A_1X$. We know that for any $x,y\in\mc A_1X$
  \[ \tau_X^*(x) \conss_{X} \tau_X^*(y) \nt x \cons_{X} y. \]
  From previous computations we already know that
  \[ \tau_X^*(x) \conss_{X} \tau_X^*(y) \eff \rho_X^*q_{1,X}^*(y) \le \rho_X^*q_{1,X}^*(x); \]
  \[ x \cons_{X} y \eff q_{1,X}^*(y) \le q_{1,X}^*(x). \]
  It then follows that
  \[ \rho_{X}^*q_{1,X}^*(y) \le \rho_{X}^*q_{1,X}^*(x) \nt q_{1,X}^*(y) \le q_{1,X}^*(x), \]
  which, together with the fact that $q_{1,X}^*$ is surjective, implies that $\rho_X^*$ is a poset-embedding, thus injective.
\end{proof}

Now intuitively, two consequence relations $\cons,\conss$ are equivalent if, and only if each one of them can be faithfully represented by the other. According to the above result, this in particular means that we must have a morphism $(\tau,\rho)$ in $\Quo$ from $q_1$ to $q_2$ where $\rho$ is injective, and another morphism $(\delta,\gamma)$ from $q_2$ to $q_1$ with $\gamma$ being injective. Set-theoretically, we know that if there are two-sided injections between two sets then they are isomorphic. Lifting this point-wise result to our category $[\Set,\SupL]$ shows that $\mc B_1$ and $\mc B_2$ must be isomorphic either. Hence, we define an \emph{equivalence} between two consequence relations as a tuple $(\tau,\delta,\rho)$ with $\rho$ being an isomorphism, such that in the below diagramme
\[
\begin{tikzcd}[row sep = 0.9cm]
  \mc A_1 \ar[r, bend left, "\tau"] \ar[d, two heads, "q_1"'] & \mc A_2 \ar[d, two heads, "q_2"] \ar[l, bend left, "\delta"] \\
  \mc B_1 \ar[r, bend left, "\rho"] & \mc B_2 \ar[l, bend left, "\rho\inv"]
\end{tikzcd}
\]
$(\tau,\rho)$ exhibits a representation of $\cons$ by $\conss$, and $(\delta,\rho\inv)$ exhibits a representation of $\conss$ by $\cons$. We have the following characterisation of equivalences between consequence relations.
\begin{lemma}
\label{lem:equivcons}
  Given two quotients $q_1: \mc A_1 \surj \mc B_1$ and $q_2:\mc A_2 \surj \mc B_2$, two morphisms $\tau : \mc A_1 \gb \mc A_2 : \delta$ induce an equivalence between the two consequence relations iff one of the two conditions holds:
  \begin{enumerate}
    \item $q_1$ exhibits $\mc B_1$ as the image of $\tau\cps q_2$ and $\delta\cps\tau\cps q_2 = q_2$;
    \item $q_2$ exhibits $\mc B_2$ as the image of $\delta\cps q_1$ and $\tau \cps \delta \cps q_1 = q_1$.
  \end{enumerate}
\end{lemma}
\begin{proof}
  It is easy to see that an equivalence between the two induced consequence relations implies both (1) and (2). Suppose (1) holds. If $q_1$ exhibits $\mc B_1$ as the image of $\tau\cps q_2$, we would then have an injective morphism $\rho : \mc B_1 \to \mc B_2$ such that
  \[ q_1 \cps \rho = \tau \cps q_2. \]
  To show $\rho$ is an isomorphism we then only need to show it is also surjective. For any set $X$, let $y$ be an element in $\mc B_2X$. Since $q_2$ is surjective, we choose $x\in\mc A_2X$ that $q_{2,X}(x) = y$. We can then compute
  \[ \rho_X(q_{1,X}(\delta_X(x))) = q_{2,X}(\tau_X(\delta_X(x))) = q_{2,X}(x) = y. \]
  The first equality holds by the fact that $q_1$ and $\rho$ is the image-factorisation of $\tau\cps q_2$; the second holds because by assumption $\delta\cps\tau\cps q_2 = q_2$. It then follows that $\rho$ is indeed surjective, hence isomorphic. Finally, we observe
  \[ q_2 \cps \rho\inv = \delta \cps \tau \cps q_2 \cps \rho\inv = \delta \cps q_2 \cps \rho \cps \rho\inv = \delta \cps q_2. \]
  This finally shows that $(\delta,\rho\inv)$ is also a morphism in $\Quo$, thus $(\tau,\delta,\rho)$ is indeed an equivalence of consequence relations. The proof of (2) is completely similar.
\end{proof}

In terms of consequence relations, we have the following result.
\begin{corollary}
\label{cor:equivcons}
  A pair of morphisms $\tau : \mc A_1 \gb \mc A_2 : \delta$ induces an equivalence between the two consequence relations $\cons,\conss$ iff the following conditions hold: For any set $X$,
  \begin{itemize}
    \item for any $x,y\in A_1X$, $x \cons_X y \eff \tau_X^*(x) \conss_X \tau_X^*(y)$;
    \item for any $z\in A_2X$, $z \cceff \tau_X^*(\delta_X^*(z))$.
  \end{itemize}
  Or equivalently, iff for any set $X$ the following conditions hold:
  \begin{itemize}
    \item for any $x,y\in A_2X$, $x \conss_X y \eff \delta_X^*(x) \cons_X \delta^*_X(y)$;
    \item for any $z\in A_1X$, $z \ceff \delta_X^*(\tau_X^*(z))$.
  \end{itemize}
\end{corollary}
\begin{proof}
  Similarly we only prove the result for the first set of conditions. From Lemma~\ref{lem:repinj} we know that $x \cons_X y \eff \tau_X^*(x) \conss \tau_X^*(y)$ implies that $\mc B_1$ is the image of $\tau \cps q_2$. On the other hand, for any $z \in A_2X$, $z \cceff \tau_X^*(\delta_X^*(z))$ simply means that
  \[ q_{2,X}^*(z) = q_{2,X}^*(\tau_X^*(\delta_X^*(z))), \]
  which shows $q_2 = \delta \cps \tau \cps q_2$. By Lemma~\ref{lem:equivcons} we then know $\tau,\delta$ induce an equivalence of consequence relations.
\end{proof}

Corollary~\ref{cor:equivcons} has shown us that given an equivalence between consequence relations $\cons$ and $\conss$, what we essentially have is two translations $\tau,\delta$ that induce faithful representations of one consequence relation by another, and vice versa. Furthermore, if you translate something back and force along $\tau$ and $\delta$, the result you get would be the same as the original element you begin with relative to the consequence relation. This has nicely summarised what we want for two equivalence relations to be equivalent in a precise technical sense.

In particular, when the consequence relation $\cons$ on $\ms P\mb F$ generated by a proof system is equivalent to the consequence relation $\vDash_\mc K$ on $\ms P\Eq$ generated by a subclass of $\mb F$-algebras $\mc K$, we say the logic on $\ms P\mb F$ is \emph{algebraisable}, and it is \emph{algebraised by $\mc K$}. There are already many examples of algebraisable systems presented in the literature, see \cite{blok1989algebraizable} for instance. Most of these examples, though not formulated in our categorical framework, can be easily seen to still be algebraisable in our extended sense. We do not further pursue any concrete examples here, but end with a discussion of how the fore mentioned Lindenbaum-Tarski algebra construction understood in our categorical framework can be seen to induce an algebraisation of certain logics.

\section{Projective Objects and Algebraisation}
In this section we provide a detailed study of an important theorem proved in \cite{blok1989algebraizable} that characterises when algebraisation is available for general concrete logical systems. The theorem says that a concrete logical system with consequence relation $\cons$ is algebraisable by some class of algebras, with the induced consequence relation $\vDash$ on pairs of formulas, if, and only if the induced lattice of theories $\mr{Th}_{\cons}$ is isomorphic to the lattice of theories $\mr{Th}_\vDash$, such that the isomorphism commutes with inverse substitutions. Our approach extends \cite{galatos&tsinakis2009consequence} as a more general categorical study of the characterisation of algebraisation, and provides a more general result.

Let's first recall the definition of a projective objects in a category. In any category $\catC$, $P$ is projective if and only if for any epimorphism $q : A \surj B$ and any morphism $s : P \to B$, there exists a lift of $s$ along $q$,
\[
\begin{tikzcd}
  & A \ar[d, two heads, "q"] \\
  P \ar[ur, dashed, "\exists"] \ar[r, "s"'] & B
\end{tikzcd}
\]
making the above diagramme commute. In other words, the Hom-functor $\catC(P,-)$ preserves epimorphisms. The fact that $\ms P$ is the free suplattice functor plays an essential role in the following result.
\begin{lemma}
\label{lem:projsuplp}
  The functor $\ms P\circ (-)$ from $[\Set,\Set]$ to $[\Set,\SupL]$ preserves projective elements. In other worlds, for any endo-functor $T : \Set \to \Set$, if $T$ is projective in $[\Set,\Set]$ then $\ms PT$ is projective in $[\Set,\SupL]$.
\end{lemma}
\begin{proof}
  Suppose $T$ is projective in $[\Set,\Set]$, and $q : \mc A \surj \mc B$ is a surjection in $[\Set,\SupL]$. Let $a$ be any morphism $a : \ms PT \to \mc B$. We know that $U_\mb Pq : A \surj B$ is a surjection in $\Set$; since $T$ is projective, it follows that there exists $l$ making the following diagramme commute,
  \[
  \begin{tikzcd}
    T \ar[r, "l"] \ar[d, "\eta_\mb P\circ\id_T"'] & A \ar[d, two heads, "U_\mb Pq"] \\
    \mb PT \ar[r, "U_\mb Pa"'] & B
  \end{tikzcd}
  \]
  Now we know that $\ms P$ is the free suplattice functor. Hence, for any set $X$, the function $l_X : TX \to AX$ will induces a unique suplattice morphism
  \[ \qsi{l_X} = \mb Pl_X \cps A_X : \mb PTX \to AX, \]
  such that $A_X : \mb PAX \to AX$ is the $\mb P$-algebra structure corresponding to the suplattice $\mc AX$. We first show the naturality of this construction. For any function $f : X \to Y$,
  \[ \qsi{l_X} \cps Af = \mb Pl_X \cps A_X \cps Af = \mb Pl_X \cps \mb PAf \cps A_Y = \mb PTf \cps \mb Pl_Y \cps A_Y = \mb PTf \cps \qsi{l_Y}. \]
  We have the second equality because $Af$ is a $\mb P$-algebra morphism; the third equality is due to naturality of $l$. It follows that we then have a natural transformation $\qsi l : \mb PT \to A$, such that each component $\qsi l_X$ is given by $\qsi{l_X}$, and $\qsi l_X$ is furthermore a suplattice morphism. It then follows that we now have a morphism in $[\Set,\SupL]$, whose underlying functor is $\qsi l$. We finally need to verify that $\qsi l$ is indeed a lift of $a$, and we can verify this at the level of underlying set. According to our definition, for any set $X$ we have
  \begin{align*}
    \qsi l_X \cps (U_\mb Pq)_X 
    &= \mb Pl_X \cps A_X \cps (U_\mb Pq)_X \\
    &= \mb Pl_X \cps \mb P(U_\mb Pq)_X \cps B_X \\
    &= \mb P\eta_{\mb P,TX} \cps \mb P(U_\mb Pa)_X \cps B_X \\
    &= \mb P\eta_{\mb P,TX} \cps \mu_{\mb P,X} \cps (U_\mb Pa)_X \\
    &= (U_\mb Pa)_X
  \end{align*}
  The first equality holds by definition of $\qsi l_X$; the second holds due to the fact that $q_X$ is a morphism between suplattices; the third holds because $l$ is a lifting of $\eta_\mb P \circ \id_T \cps U_\mb Pa$ along $U_\mb Pq$; the fourth is again because $a_X$ is a morphism between suplattices and the $\mb P$-algebra structure on the free suplattice $\ms PT$ is given by $\mu_{\mb P,X}$; the final equality holds due to the triangular identity of the power set monad. Such a computation shows that we indeed have a lift of $a$ along $q$. Thus, $\ms PT$ is also projective in $[\Set,\SupL]$.
\end{proof}
Lemma~\ref{lem:projsuplp} implies that to show a functor $\ms PT$ is projective in $[\Set,\SupL]$, it is enough to show that $T$ is projective in $[\Set,\Set]$. A very important class of projective objects in $[\Set,\Set]$ are polynomial functors. It is well-known in the literature that polynomial functors are projective. For sake of completeness we provide the proof here. We first show that representable functors are projective.
\begin{lemma}
\label{lem:repreproj}
  Any representable functor $\yon^X$ in $[\Set,\Set]$ for any set $X$ is projective.
\end{lemma}
\begin{proof}
  By Yoneda lemma, a morphism $a : \yon^X \to A$ is the same as an element in $A(X)$. Given any surjection $q : A \surj B$ and any morphism $s : \yon^X \to B$, or equivalently $s \in B(X)$, by surjectivity we can find an element $t \in A(X)$ such that $q_X(t) = s$. By Yoneda again, it follows that the following diagramme commutes,
  \[
  \begin{tikzcd}
    & A \ar[d, "q"] \\
    \yon^X \ar[ur, "t"] \ar[r, "s"'] & B 
  \end{tikzcd}
  \]
  Thus, $\yon^X$ is projective.
\end{proof}

It is also well-known that arbitrary coproducts of projective objects is again projective.
\begin{lemma}
\label{lem:coprodproj}
  If $P_i$ is projective for any $i\in I$, then if the coproducts $\sum_{i\in I}P_i$ exists, $\sum_{i\in I}P_i$ is also projective.
\end{lemma}
\begin{proof}
  Given an epimorphism $q : A \to B$, a morphism from $\sum_{i\in I}P_i \to B$ is of the form
  \[ [(f_i)_{i\in I}] : \sum_{i\in I}P_i \to B, \]
  with each $f_i : P_i \to B$. Now since $P_i$ is projective, there exists a lift $l_i : P_i \to A$ such that
  \[ l_i \cps q = f_i. \]
  Hence, we can construct the lift for $[(f_i)_{i\in I}]$ as $[(l_i)_{i\in I}]$, which implies
  \[ [(l_i)_{i\in I}] \cps q = [(l_i \cps q)_{i\in I}] = [(f_i)_{i\in I}]. \]
  Thus $\sum_{i\in I}P_i$ is also projective.  
\end{proof}

\begin{corollary}
\label{cor:polyproj}
  All polynomial functors in $[\Set,\Set]$ are projective.
\end{corollary}
\begin{proof}
  By definition, polynomial functors in $[\Set,\Set]$ are exactly coproducts of representable functors in $[\Set,\Set]$.
\end{proof}

Recall in Remark~\ref{rem:synpolymonad} we've mentioned that the syntactic monad $\mb F$ is a free monad on a polynomial functor $H$, which makes itself a polynomial monad. Explicitly, the polynomial functor $H$ is given by
\[ H = \sum_{\star\in\Sigma} \yon^{\ar(\star)}. \]
We can also see more directly that $\mb F$ is a polynomial functor, since by definition it can be described as the following coproduct
\[ \mb F \cong \yon + H + H \circ H + \cdots. \]
For any set $X$, $HX$ can be identified as the set of all terms of the form $\star(x_1,\cdots,x_n)$, with $\star\in\Sigma$ and $x_1,\cdots,x_n\in X$. Hence, the set of all formulas, or $\Sigma$-terms, $\mb FX$ is then naturally identified as follows
\[ \mb FX \cong X + HX + HHX + \cdots. \]
Since the full subcategory $\Poly$ of polynomial functors in $[\Set,\Set]$ is closed under products, coproducts and compositions (see \cite{spivak2021poly}), it follows that $\mb F$ is indeed a polynomial functor, and so does $\Eq$, or more generally the functor $\Seq_{n,m} \cong \mb F^n \times \mb F^m$ that encodes $(n,m)$-sequents of formulas. Hence, By Corollary~\ref{cor:polyproj}, $\mb F$, as well as $\Eq$ and $\Seq_{n,m}$, is projective in $[\Set,\Set]$; and further by Lemma~\ref{lem:projsuplp} it follows that $\ms P\mb F$ is also projective in $[\Set,\SupL]$. 


  

If $\mc P$ is a projective object in $[\Set,\SupL]$, then given any quotients $q_1 : \mc P \surj \mc Q$ and $q_2 : \mc A \surj \mc B$, if we have a morphism $\rho : \mc Q \to \mc B$, then by projectivity of $\mc P$ it follows that there must be a lifting of $q_1 \cps \rho$ along $q_2$, viz. a morphism $\tau$ making the following diagramme commute,
\[
\begin{tikzcd}
  \mc P \ar[d, two heads, "q_1"'] \ar[r, dashed, "\tau"] & \mc A \ar[d, two heads, "q_2"] \\
  \mc Q \ar[r, "\rho"'] & \mc B
\end{tikzcd}
\]
This implies that once we have a morphism $\rho$ from $\mc Q$ to $\mc B$, we are guaranteed to find a morphism $\tau : \mc P \to \mc A$, making the above diagramme a morphism in $\Quo$. Furthermore, if both quotients are quotients of projective objects, then we have the following corollary of equivalence of consequence relations.
\begin{corollary}
\label{cor:isoalgebraisation}
  For quotients $q_1 : \mc P_1 \surj \mc Q_1$, $q_2 : \mc P_2 \surj \mc Q_2$ with $\mc P_1,\mc P_2$ being projective in $[\Set,\SupL]$, they are equivalent as consequence relations if, and only if $\mc Q_1$ and $\mc Q_2$ are isomorphic.
\end{corollary}
\begin{proof}
  The only if direction is easy. Suppose $\mc Q_1$ and $\mc Q_2$ are isomorphic with the isomorphism $\rho : \mc Q_1 \cong \mc Q_2$, then by projectivity of $\mc P_1,\mc P_2$, there exists morphisms $\tau : \mc P_1 \gb \mc P_2 : \delta$, such that
  \[ \tau \cps q_2 = q_1 \cps \rho, \quad \delta \cps q_1 = q_2 \cps \rho\inv. \]
  By definition, it follows that the two consequence relations are equivalent.
\end{proof}

\bibliography{mybib}{}

\begin{thebibliography}{}

\bibitem[Blok and J{\'o}nsson, 2006]{blok2006equivalence}
Blok, W.~J. and J{\'o}nsson, B. (2006).
\newblock Equivalence of consequence operations.
\newblock {\em Studia Logica}, 83(1):91--110.

\bibitem[Blok and Pigozzi, 1989]{blok1989algebraizable}
Blok, W.~J. and Pigozzi, D. (1989).
\newblock {\em Algebraizable logics}, volume~77.
\newblock American Mathematical Soc.

\bibitem[Borceux, 1994]{borceux1994handbook1}
Borceux, F. (1994).
\newblock {\em Handbook of categorical algebra: volume 1, Basic category
  theory}, volume~1.
\newblock Cambridge University Press.

\bibitem[Borceux, 2005]{borceux2005internalaction}
Borceux, Francis, J. G. Z. K. G.~M. (2005).
\newblock Internal object actions.
\newblock {\em Commentationes Mathematicae Universitatis Carolinae},
  46(2):235--255.

\bibitem[Galatos and Tsinakis, 2009]{galatos&tsinakis2009consequence}
Galatos, N. and Tsinakis, C. (2009).
\newblock Equivalence of consequence relations: an order-theoretic and
  categorical perspective.
\newblock {\em The Journal of Symbolic Logic}, 74(3):780--810.

\bibitem[Gambino and Kock, 2013]{gambino2013polynomial}
Gambino, N. and Kock, J. (2013).
\newblock Polynomial functors and polynomial monads.
\newblock In {\em Mathematical proceedings of the cambridge philosophical
  society}, volume 154, pages 153--192. Cambridge University Press.

\bibitem[Halbach and Leigh, 2021]{halbachleigh2021}
Halbach, V. and Leigh, G. (2021).
\newblock The road to paradox: A guide to syntax, truth, and modality.
\newblock To be published.

\bibitem[Hyland and Power, 2007]{hyland&poewr2007algandmonad}
Hyland, M. and Power, J. (2007).
\newblock The category theoretic understanding of universal algebra: Lawvere
  theories and monads.
\newblock {\em Electronic Notes in Theoretical Computer Science}, 172:437--458.

\bibitem[Johnstone, 1975]{johnstone1975adjointlifting}
Johnstone, P.~T. (1975).
\newblock Adjoint lifting theorems for categories of algebras.
\newblock {\em Bulletin of the London Mathematical Society}, 7(3):294--297.

\bibitem[Johnstone, 1982]{johnstone1982stone}
Johnstone, P.~T. (1982).
\newblock {\em Stone spaces}, volume~3.
\newblock Cambridge university press.

\bibitem[Joyal and Tierney, 1984]{joyal1984extension}
Joyal, A. and Tierney, M. (1984).
\newblock {\em An extension of the Galois theory of Grothendieck}, volume 309.
\newblock American Mathematical Soc.

\bibitem[Kelly, 1980]{kelly1980unifiedfree}
Kelly, G.~M. (1980).
\newblock A unified treatment of transfinite constructions for free algebras,
  free monoids, colimits, associated sheaves, and so on.
\newblock {\em Bulletin of the Australian Mathematical Society}, 22(1):1--83.

\bibitem[Mac~Lane, 2013]{maclane2013categories}
Mac~Lane, S. (2013).
\newblock {\em Categories for the working mathematician}, volume~5.
\newblock Springer Science \& Business Media.

\bibitem[Manes, 2003]{manes2003monads}
Manes, E. (2003).
\newblock Monads of sets.
\newblock In {\em Handbook of algebra}, volume~3, pages 67--153. Elsevier.

\bibitem[Spivak and Nelson, 2021]{spivak2021poly}
Spivak, D. and Nelson, N. (2021).
\newblock {\em Polynomial Functors: A General Theory of Interaction}.
\newblock Topos Institute.
\newblock Sept. 8th version.

\bibitem[Tarski, 1928]{tarski1928remarques}
Tarski, A. (1928).
\newblock Remarques sur les notions fondamentales de la m{\'e}thodologie des
  math{\'e}matiques.
\newblock In {\em Annales de la Soci{\'e}t{\'e} Polonaise de Math{\'e}matique},
  volume~7, pages 270--272.

\end{thebibliography}
\bibliographystyle{apalike}

\end{document}